\documentclass[12pt,a4paper]{amsart}
\usepackage[utf8x]{inputenc}
\usepackage[T1]{fontenc}
\usepackage[english]{babel}
\usepackage{amsmath}
\usepackage{amsfonts}
\usepackage{amssymb}
\usepackage{amsthm}
\usepackage[colorlinks=true, allcolors=blue]{hyperref}
\usepackage{color}
\usepackage{graphicx}
\usepackage{subcaption}
\usepackage{afterpage}
\usepackage{mathtools}
\usepackage{pdfpages}
\usepackage{bm}
\usepackage{appendix}
\usepackage{cleveref}
\usepackage{xfrac}

\numberwithin{equation}{section}

\textwidth 
           460pt

\newtheorem{theorem}{Theorem}[section]
\newtheorem{lemma}[theorem]{Lemma}
\newtheorem{proposition}[theorem]{Proposition}

\theoremstyle{definition}
\newtheorem{definition}[theorem]{Definition}
\theoremstyle{remark}
\newtheorem{remark}[theorem]{Remark}
\theoremstyle{remark}
\newtheorem{assumption}[theorem]{Assumption}

\newcommand{\inn}{\text{in }}
\newcommand{\onn}{\text{on }}

\newcommand{\spann}{\mathrm{span}}

\DeclareMathOperator{\im}{im}

\title{Internal doubly periodic gravity-capillary waves with vorticity
}
\author{Douglas Svensson Seth}
\thanks{Department of Mathematical Sciences, Norwegian University of Science and Technology, \text{douglas.s.seth@ntnu.no}}
\begin{document}
\begin{abstract}
		We consider a multi-fluid system with several free interfaces. For this system we prove existence of three-dimensional steady gravity-capillary waves with non-zero vorticity. We obtain non-zero vorticity by prescribing the relative velocity fields to be Beltrami fields, for which the vorticity and velocity are parallel. The main result is a multi-parameter bifurcation result for small amplitude waves given in two variants: a first theorem guaranteeing existence under some general parameter assumptions; and a second specific but less exhaustive theorem, for which the assumptions may be explicitly verified, yielding the existence of both in-phase and off-phase motions in the different layers. The proof relies on an implicit function theorem corresponding to multi-parameter bifurcation. This theorem is presented in an appendix as an abstract result that can be applied directly to other problems. 
\end{abstract}
\maketitle
\section{Introduction}
	In this paper we consider $ n+1 $ immiscible fluids separated by $ n $ free boundaries. All the fluids are contained within the domain
	\[
	\Omega=\{\bm{x}=(\bm{x}',z)=(x,y,z)\in\mathbb{R}^2\times \mathbb{R}:0<z<d_{n+1}\}.
	\]
	which itself is separated into the $ n+1 $ layers given by
	\[
	\Omega_j=\{(\bm{x}',{z})\in\mathbb{R}^2\times\mathbb{R}:\eta_{j-1}+d_{j-1}<z<\eta_j+d_{j}\},\qquad j=1,\ldots,n+1,
	\]
	for $n$ different interface profiles $\eta_j$ ($ \eta_0=\eta_{n+1}=0 $ as well as $d_0=0$); see \cref{fig:domain}. 
	\begin{figure}
		\begin{center}
		\includegraphics[scale=1]{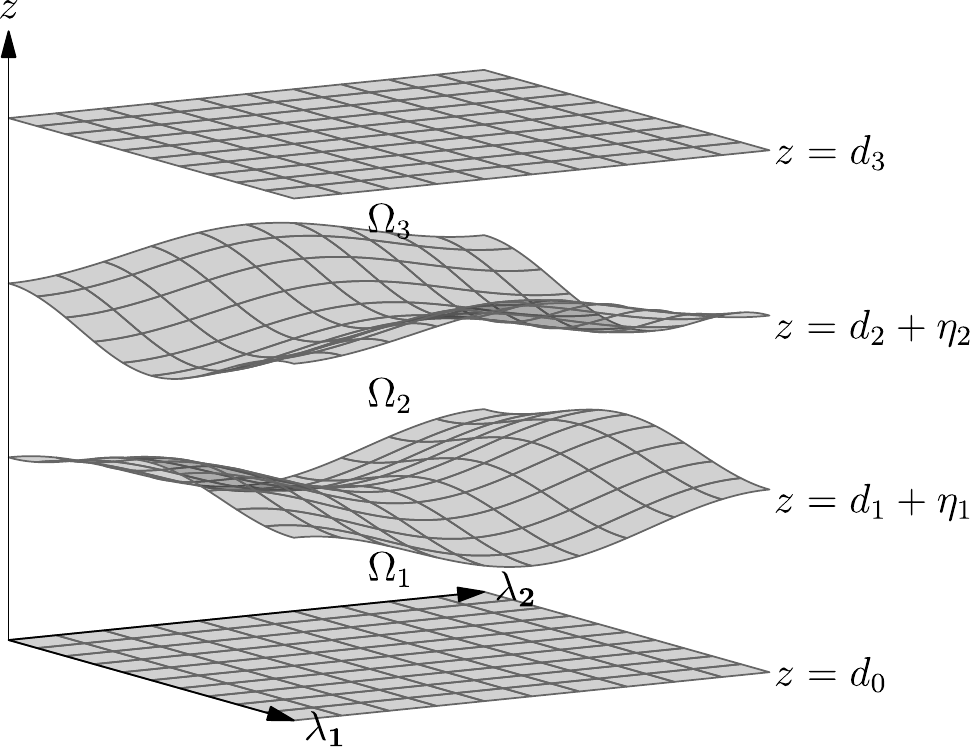}
		\caption{The domain $ \Omega $ for $ n=2 $. In this case it is separated into three parts, $ \Omega_1 $, $ \Omega_2 $, $ \Omega_3 $, by the interfaces at $ z=d_1+\eta_1 $ and $ z=d_2+\eta_2 $. This figure is restricted to `one period' of $ \Omega $ determined by the vectors $ \bm{\lambda}_1 $ and $ \bm{\lambda}_2 $.}
		\label{fig:domain}
		\end{center}
	\end{figure}
	$ \Omega_j $ contains the $ j $:th fluid. In each layer we assume that the fluid has constant density $ \rho_j $ and that the velocity field $ \bm{w}^{(j)} $ of the fluid satisfies the Euler equations with an external gravitational force $ \bm{g}=(0,0-g) $. In the remainder of the paper we assume that the densities always satisfy the condition $ \rho_1>\rho_2>\ldots>\rho_{n+1} $.
	We also work under the traveling wave assumption, that is, the velocity fields and interface profiles are time independent in some frame of reference moving with constant speed $\bm{c}$. In other words, the waves travel with speed $\bm{c}$.
	Instead of working directly with $\bm{w}^{(j)}$ we will work with the relative velocity field in the frame of reference moving with the waves $\bm{u}^{(j)}$. This relative velocity field is obtained by setting $\bm{u}^{(j)}(\bm{x})=\bm{w}^{(j)}(\bm{x})-\bm{c}$ and it solves the steady Euler equations
	\begin{align}
	(\bm{u}^{(j)}\cdot\nabla)\bm{u}^{(j)}&=-\frac{1}{\rho_j}\nabla{p}^{(j)}+\bm{g},\label{eq:eulermoment}\\
	\nabla\cdot \bm{u}^{(j)}&=0.\label{eq:eulerdiv}
	\end{align}
	From now on we will just refer to $\bm{u}^{(j)}$ as the velocity field. We assume that the velocity fields are \emph{Beltrami fields}, which means that the vorticity is parallel to the velocity. In particular, we shall assume that the velocity field is a \emph{strong Beltrami field} in each layer. Here \emph{strong} means that the proportionality factor between velocity and vorticity is a constant. In other words, $ \nabla \times\bm{u}^{(j)}=\alpha_j\bm{u}^{(j)} $ in $ \Omega_j $ for some constant $\alpha_j$. However, we do allow $ \alpha_j\neq\alpha_i $ for $ j\neq i $. For mathematical reasons this is a very suitable vorticity assumption when working with the Euler equations. The identity
	\[
		\nabla \frac{1}{2}\vert \bm{u}\vert^2=(\bm{u}\cdot\nabla)\bm{u}+\bm{u}\times(\nabla \times \bm{u})
	\]
	means that the momentum equation of the Euler equations, \cref{eq:eulermoment}, is satisfied with pressure given by
	\begin{equation}\label{eq:pressure}
	p^{(j)}=-\rho_{j}\left(\frac{1}{2}\vert\bm{u}^{(j)}\vert^2+gz\right)+Q_j
	\end{equation}
	if $ \bm{u}^{(j)} $ is a Beltrami field, for an arbitrary constant $ Q_j $; like in the irrotational case, which is the special case $ \alpha_j=0 $.
	Moreover, we introduce two linearly independent vectors $ \bm{\lambda}_1$ and $ \bm{\lambda}_2 $, which allows us to define the lattice
	\[
	\Lambda=\{\bm{\lambda}=m_1\bm{\lambda}_1+m_2\bm{\lambda}_2:m_1,m_2\in\mathbb{Z}\}.
	\]
	We shall assume that our solutions, whence the waves, are periodic with respect to this lattice; see \cref{fig:domain}. For future reference we also introduce the dual lattice
	\[
		\Lambda'=\{\bm{k}=m_1\bm{k}_1+m_2\bm{k}_2:m_1,m_2\in\mathbb{Z}\}
	\]
	where $\bm{\lambda}_i\cdot \bm{k}_j=2\pi\delta_{ij}$. For our analysis it is suitable to express the vectors in the dual lattice in polar form, so we let $\bm{k}=k(\cos(\gamma),\sin(\gamma))$ for some general vector $\bm{k}\in\Lambda'$ and in particular $\bm{k}_i=k_i(\cos(\gamma_i),\sin(\gamma_i))$, $i=1,2$. We also denote `one period' with respect to $\Lambda$ by
	\[
	\Gamma=\{(x,y)\in\mathbb{R}^2: (x,y)=a\bm{\lambda}_1+b\bm{\lambda}_2, 0<a<1,0<b<1\}.
	\]
	and one period of $ \Omega_j $ by
	\[
		\hat{\Omega}_j=\{(x,y,z)\in\Omega_j: (x,y)\in \Gamma\}.
	\]	
	In summary, the velocity fields satisfy the equations
	\begin{align}
		{\nabla} \times {\bm{u}}^{(j)}&=\alpha_j {\bm{u}}^{(j)}	&&\inn \Omega_j,\label{eq:curl}\\
		{\nabla} \cdot {\bm{u}}^{(j)}&=0 && \inn \Omega_j,\label{eq:div}
	\end{align}
	for $ j=1,\ldots,n+1 $, and since the fluids are assumed to be immiscible we get the boundary conditions 
	\begin{align}
		\bm{n}\cdot {\bm{u}}^{(j)}&=0 && \onn \partial\Omega_j,\label{eq:kin_bdry}
	\end{align}
	for $j=1,\ldots,n+1$ and the (upwards) unit normal $ \bm{n} $ of $\partial \Omega_j $.
	Moreover, the pressure difference between $ \Omega_j $ and $ \Omega_{j+1} $ at their shared boundary satisfies the Young--Laplace law
	\begin{align*} 
		p_{j+1}-p_{j}&=-\sigma_j\nabla \cdot \bm{n} &&\onn z=\eta_{j}+d_j,
	\end{align*}
	for $j=1,\ldots,n$. Here $ \sigma_j $ denotes an interfacial tension parameter.
	We can substitute the pressure using \cref{eq:pressure} and obtain
	\begin{align}
		\nonumber\rho_j\left(\frac{1}{2}\vert{\bm{u}}^{(j)}\vert^2+g(\eta_{j}+d_j)\right)-\rho_{j+1}&\left(\frac{1}{2}\vert{\bm{u}}^{(j+1)}\vert^2+g(\eta_{j}+d_j)\right)\\
		-\sigma_{j}{\nabla} \cdot \left(\frac{{\nabla} \eta_{j}}{\sqrt{1+\vert {\nabla} \eta_{j}\vert^2}}\right)&=Q_j-Q_{j+1}&&\onn z=\eta_{j}+d_j,\label{eq:dyn_bdry}
	\end{align}
	for $j=1,\ldots,n$. We normalize the pressure, that is, choose the $Q_j $, in \Cref{sec:triv_sol}.
	The \crefrange{eq:curl}{eq:dyn_bdry} constitute a free boundary problem with several undetermined interfaces $\eta_j$. 
	
	In the present paper we will also allow $ \rho_{n+1}=0 $, to be able to capture interaction between surface waves and internal waves. Note that with $ \rho_{n+1}=0 $ and $ n=1 $ we recover the classical water wave problem for surface waves. Setting $ \rho_{n+1}=0 $ completely decouples the problem from the uppermost layer, so we introduce
	\[
	m=\left\{
	\begin{aligned}
	&n+1 && \text{if  } \rho_{n+1}>0,\\
	&n && \text{if  } \rho_{n+1}=0,
	\end{aligned}\right.
	\]
	to be able to handle both cases simultaneously.
	Moreover, we will use the following notation
	\begin{align*}
	&\bm{u}=(\bm{u}^{(1)},\ldots,\bm{u}^{(m)}),
	&&\bm{\eta}=(\eta_1,\ldots \eta_n), &&&\bm{\sigma}=(\sigma_1,\ldots,\sigma_n),\\ &\bm{\rho}=(\rho_1,\ldots,\rho_{m}), &&\bm{\alpha}=(\alpha_1,\ldots,\alpha_{m}),
	&&&\bm{d}=(d_1,\ldots,d_{m}).
	\end{align*}
	\subsection{Background}
	This type of free boundary problem has been extensively studied in two dimensions. Especially the case with two fluids separated by one free interface, that is, $ n=1 $; see for example \cite{Amick_1989, Dias_1996, Nilsson_2016, Sun_1993}. One of the common applications of a multi-fluid system like this is the study of internal waves in an ocean stratified by, for example, temperature or salinity; see \cite[section 7]{Haziot_2022} for an overview. However, the two-layer model that is usually studied gives an idealized version of an ocean with almost constant, but different, densities in two layers separated by a sharp density gradient. If the gradient is sufficiently sharp then it is intuitive to approximate it with a single layer given by a free interface; this has also been rigorously justified under certain conditions by Chen and Walsh \cite{Chen_2015}. Naturally with this approximation we may lose some of the finer detail from a varying density, which has been studied numerically by Vanden-Broeck and Turner \cite{Vanden_Broeck_1992}. To recapture some of that detail is one of the reason we study more than two layers. This have been proposed before, by for example Rusås and Grue \cite{Rusas_2002}, and there is at least numerical support that several layers allows us to recapture some phenomena lost by collapsing the change in density to a single layer; see for example Nakayama and Lamb \cite{Nakayama_2020}. 
	
	In three dimensions most existing results treat only surface waves. The first rigorous existence result for doubly periodic waves on a symmetric lattice is due to Reeder and Shinbrot \cite{Reeder_1981}; this was extended to general lattices Craig and Nicholls \cite{Craig_2000}. Both these results are based on bifurcation theory, albeit Craig and Nicholls also employs a variational approach to the problem. Another method that has proven useful is that of spacial dynamics; first used for the water wave problem in three dimensions by Groves and Mielke \cite{Groves_2001} and Groves and Haragus \cite{Groves_2003}. All these are results for gravity-capillary waves. This is due to a small-divisor problem that appears for pure gravity waves in three dimensions, making the problem in many ways easier with surface tension. The existence of surface gravity waves has however been proven by Iooss and Plotnikov \cite{Iooss_2010,Iooss_2009} using Nash-Moser techniques. In the present paper interfacial tension is included as it resolves the small-divisor problem. Similarly, in the dynamical setting nonzero interfacial tension resolves the Kelvin–Helmholtz instability for high frequencies; see Lannes \cite{Lannes_2013}. Since our result is valid for arbitrarily small, although non-zero, interfacial tension we should be able to provide solutions were the interfacial tension brings stability without having a large impact on the wave profiles. For internal waves in three dimensions there is one existence result relying on spacial dynamics, due to Nilsson \cite{Nilsson_2019}.
	
	The problem studied in this paper includes vorticity, due to the assumption that the velocity fields are Beltrami fields. Waves with vorticity have been studied extensively in two dimensions, but in three dimensions the results are more sparse. For surface waves in three dimensions with vorticity there is a non-existence result for constant vorticity due to Wahlén \cite{Wahlen_2014}. There are also two existence results in the same setting; the first due to Lokharu, Seth and Wahlén \cite{Lokharu_2020}, which is similar in nature to the present contribution in that the velocity is assumed to be a Beltrami field; the second, by Seth, Varholm and Wahlén \cite{Seth_2022}, is based on a different vorticity assumption, which is inspired by a mathematically equivalent problem in plasma physics. There is one result considering internal waves in three dimensions with vorticity. Chen, Fan, Walsh and Wheeler \cite{Chen_2022} show that if the vorticity is constant then the solutions are very restricted.
	
	\subsection{Main result and structure of the article}
	In the present paper we obtain an existence result for three dimensional, internal waves with nonzero vorticity. The main result requires additional technical definitions to be stated precisely, but we give a summarized version here.
	\begin{theorem}\label{thm:main_intro}
	Under appropriate assumptions on the parameters, $ \bm{\sigma},\bm{\rho},\bm{\alpha},\bm{d}, \bm{\lambda}_1,\bm{\lambda}_2 $, there exists an $ \epsilon>0 $, such that for every $\bm{t}=(t_1,t_2)\in B_{\epsilon}(0)=\{\bm{t}\in \mathbb{R}^2:\vert \bm{t}\vert <\epsilon\}$ there exist solutions $\bm{u}^{(j)}(\bm{t})$, $j=1,\ldots,m$, and $\eta_j(\bm{t})$, $ j=1,\ldots,n $ to \crefrange{eq:curl}{eq:dyn_bdry}. Moreover, the interface profiles are given by
	\[
	\eta_j=t_1\hat{\eta}_j(\bm{k}_1)\cos(\bm{k}_1\cdot \bm{x}')+t_2\hat{\eta}_j(\bm{k}_2)\cos(\bm{k}_2\cdot\bm{x}')+\mathcal{O}(\vert \bm{t}\vert^2).
	\]
	for some real numbers $ \hat{\eta}_j(\bm{k}_1) $, $ \hat{\eta}_j(\bm{k}_2) $, $ j=1,\ldots,n $ depending on the parameters.
	\end{theorem}
	The detailed version of this result is given in \Cref{thm:exist}. The assumptions referenced in the theorem is not obviously satisfied, and do indeed fail for some parameter values. For this reason we show that there exists a non-empty subset of the parameter space where the assumptions are satisfied, with the corresponding existence results given in \Cref{thm:existsn1,thm:existssmallvort}. These results cover a large part of the parameter space, but a complete characterization lies beyond the scope of this paper.
	
	The overall structure of the proof of \Cref{thm:main_intro} is reminiscent of \cite{Lokharu_2020} and relies on multi-parameter bifurcation. To this end, we finish the introduction by defining the trivial solutions that the non-trivial solutions in \Cref{thm:main_intro} bifurcate from. In \cref{sec:Func_set} we set up suitable function spaces for the remaining analysis. In \Cref{sec:flat} we change coordinates to a flattened domain and reduce the problem to a single equation for the free interfaces. In \Cref{sec:lin_prob} we study the linearized version of this reduced problem. With the results from \Cref{sec:lin_prob} we state a purely algebraic assumption used for the main existence result. Both the assumption and main existence result are given, and in the latter case proved, in \Cref{sec:existence}. We end by showing that this assumption is satisfied in certain subsets of the parameter space in \Cref{sec:assumption}. We focus on two cases: limiting the number of layers to two, that is $  n=1 $, and considering weak vorticity, that is $ \vert \bm{\alpha}\vert\ll 1 $. However, we also include an informal discussion and some examples of other cases in \Cref{sec:largeVorticity}. \Cref{sec:AppendixBif} contains a multi-parameter bifurcation result that is integral in the proof of the main theorem. Similar techniques have been used repeatedly in the literature (for example in \cite{Ehrnstrom_2011,Ehrnstrom_2019,Lokharu_2020}), but only for special cases. Here, on the other hand, we have abstracted the previously used ideas and present them in a general result. This will give a handy tool for future research, since it can be directly applied to similar problems.
	
	\subsection{Trivial solutions}\label{sec:triv_sol}
	For flat interfaces, $\bm{\eta}=0$, we can find laminar flows that are explicit solutions with nonzero velocity by considering the basis functions
	\begin{align*}
		\bm{V}^{(j)}_1(z)&=(\cos(\alpha_j z),-\sin(\alpha_j z),0),\\
		\bm{V}^{(j)}_2(z)&=(\sin(\alpha_j z),\cos(\alpha_j z),0).\\
	\end{align*}
	With any $ \bm{r}=(r_1,\ldots,r_m)\in [0,\infty)^m$ and $\bm{\theta}=(\theta_1,\ldots,\theta_m)\in (\mathbb{R}/2\pi\mathbb{Z})^m$ we can construct a solution $\bm{U}=(\bm{U}^{(1)},\ldots,\bm{U}^{(m)})$ to \crefrange{eq:curl}{eq:kin_bdry} which is given by 
	\[
	\bm{U}^{(j)}(z)=r_j\cos(\theta_j)\bm{V}_1^{(j)}(z)+r_j\sin(\theta_j)\bm{V}_2^{(j)}(z)=r_j(\cos(\theta_j-\alpha_jz),\sin(\theta_j-\alpha_jz),0).
	\]
	Now we pick the $ Q_j $ in such a way that these velocity fields also satisfy \cref{eq:dyn_bdry}. This can be done by setting 
	\[ 
		Q_j=\rho_j\frac{r_j^2}{2}+C_j,\qquad j=1,\ldots,m,
	\]
	where the $ C_j $ satisfy
	\[ 
		C_j-C_{j+1}=\rho_jgd_j-\rho_{j+1}gd_j,\qquad j=1,\ldots,n .
	\]
	This leaves us with one degree of freedom in the pressure, that is, we can add the same constant to all $ C_{j} $. This does not effect the mathematical problem, though, and we can remove this freedom by simply setting $ C_1=0 $. To further decrease the degrees of freedom we will only keep $ r_1 $ and $ \theta_1 $ as free parameters and define the other $ r_j $ and $ \theta_j $ as functions of these through the recursive relationships
	\[ 
	\theta_{j+1}=\theta_j-\alpha_jd_j+\alpha_{j+1}d_j,\qquad\text{and}\qquad r_{j+1}=r_j.
	\]
	\emph{These are the relations that leaves us with a continuous trivial solution if we glue together all $ \bm{U}^{(j)} $ to one function in $ \Omega $}. However, there is no mathematical requirement for continuity and all $ r_j $ and $ \theta_j $ could be kept as free parameters.
	In fact, keeping them all as free parameters could potentially let us handle bifurcation from a point were the kernel of the linearization has higher dimension. However, in this paper we restrict ourselves to a two-dimensional kernel, which means two bifurcation parameters $ r_1 $ and $ \theta_1 $ are sufficient. For notational simplicity we drop the indices from $ r_1 $ and $ \theta_1 $, and let $ \bm{\tau}= (r,\theta)\in[0,\infty)\times (\mathbb{R}/2\pi\mathbb{Z})\eqqcolon \mathfrak{Z} $. With these choices \cref{eq:dyn_bdry} become
	\begin{align*}
	\rho_j\left(\frac{1}{2}\vert{\bm{u}}^{(j)}\vert^2+g\eta_{j}\right)-\rho_{j+1}&\left(\frac{1}{2}\vert{\bm{u}}^{(j+1)}\vert^2+g\eta_{j}\right)\\
	-\sigma_{j}{\nabla} \cdot \left(\frac{{\nabla} \eta_{j}}{\sqrt{1+\vert {\nabla} \eta_{j}\vert^2}}\right)&=(\rho_j-\rho_{j+1})\frac{r^2}{2}&&\onn z=\eta_{j}+d_j,
	\end{align*}
	which is satisfied by $ \bm{U} $ and $ \bm{\eta}=0 $.
	
	Now we impose the following integral conditions
	\begin{align}\label{eq:int_cond}
		\int_{\hat{\Omega}_j} {u}_i^{(j)}d{\bm{x}}&=\int_{\hat{\Omega}_j} U_i^{(j)}[\bm{\tau}]d{\bm{x}}&&i=1,2,
	\end{align}
	for $j=1,\ldots,m$, in addition to the equations \crefrange{eq:curl}{eq:dyn_bdry}. They will allow us to  find unique  $ \bm{u}^{(j)} $ solving \crefrange{eq:curl}{eq:kin_bdry} for given $ \bm{\eta} $ and $\bm{\tau}$, which allows us to reduce the problem to \cref{eq:dyn_bdry} with $ \bm{\eta} $ and $ \bm{\tau} $ as unknowns. In particular, for $ \bm{\eta}=0 $ we obtain $ \bm{u}^{(j)}=\bm{U}^{(j)}[\bm{\tau}] $, that solve \cref{eq:dyn_bdry} for all $ \bm{\tau}\in \mathfrak{Z} $.
	
	\section{Functional analytic setting}\label{sec:Func_set}
		We work in the real valued Hölder spaces $C^{k,\delta}_{per}(X)$, where $ X=\overline{\Omega}_j$ for the velocity fields and $ X=\mathbb{R}^2 $ for the interfaces. These are Banach spaces equipped with the norm
		\[
			\Vert f\Vert_{C^{k,\delta}_{per}(X)}=\max_{\vert \mu\vert\leq k}\sup_{\bm{x}\in X}\vert \partial^\mu f(\bm{x})\vert+ \max_{\vert \mu\vert= k}\sup_{\bm{x}\neq\bm{y}\in\overline{\Omega}_j}\frac{\vert \partial^\mu f(\bm{x})-\partial^\mu f(\bm{y}) \vert}{\vert\bm{x}-\bm{y}\vert^\delta},
		\]
		where $\delta$ is a fixed number in $(0,1)$. The subscript \emph{per} denotes that they are restricted to functions that are periodic with respect to $ \Lambda $.
		Moreover, we want the velocity fields and interfaces to satisfy certain symmetry conditions.
		By a subscript $e$ we denote functions that are even with respect to $\bm{x}'$ and by a subscript $o$ we denote functions that are odd with respect to $\bm{x}'$.
		That is,
		\begin{align*}
			C^{k,\delta}_{per,e}(\overline{\Omega}_j)&=\{f\in C^{k,\delta}_{per}(\overline{\Omega}_j):f(-\bm{x}',z)=f(\bm{x}',z)\},\\
			C^{k,\delta}_{per,o}(\overline{\Omega}_j)&=\{f\in 	C^{k,\delta}_{per}(\overline{\Omega}_j):f(-\bm{x}',z)=-f(\bm{x}',z)\},
		\end{align*}
		and likewise for $C^{l,\delta}_{per}(\mathbb{R}^2)$. After the flattening transform in the next section it is clear that the natural relation between $ k $ and $ l $ is $ l=k+1 $. Therefore we seek solutions in $\bm{u}^{(j)}\in (C^{1,\delta}_{per,e}(\overline{\Omega}_j))^2\times C^{1,\delta}_{per,o}(\overline{\Omega}_j)\eqqcolon \mathfrak{X}_j$ and $\eta_j\in C^{2,\delta}_{per,e}(\mathbb{R}^2)\eqqcolon\mathfrak{Y}_j$, which is the lowest regularity in these spaces that allow solutions in the classical sense. We also let
		\begin{align*}
			\mathfrak{X}&\coloneqq\mathfrak{X}_1\times\ldots \times\mathfrak{X}_m,\\
			\mathfrak{Y}&\coloneqq\mathfrak{Y}_1\times\ldots \times\mathfrak{Y}_n,
		\end{align*}
		so that $ \bm{u}\in \mathfrak{X} $ and $ \bm{\eta}\in \mathfrak{Y} $. For future reference we also introduce the lower regularity spaces $ \mathfrak{W}_j\coloneqq C^{0,\delta}_{per,e}(\mathbb{R}^2) $ and $ \mathfrak{W}\coloneqq \mathfrak{W}_1\times\ldots\times \mathfrak{W}_n $.
		
		Due to the periodicity and symmetries we can express $ \bm{u}^{(j)}\in \mathfrak{X}_j $ as a Fourier series
		\[ 
			\bm{u}(\bm{x}',z)=\sum_{\bm{k}\in\Lambda'} (\hat{u}_{1}(z,\bm{k}),\hat{u}_{2}(z,\bm{k}),\hat{u}_{3}(z,\bm{k}))e^{i\bm{k}\cdot \bm{x}'},
		\]
		where $ \hat{u}_{i}(z,\bm{k}) $ are real valued and satisfy $\hat{u}_{i}(z,\bm{k})=\hat{u}_{i}(z,-\bm{k})$, for $ i=1,2 $ and $ \hat{u}_{3}(z,\bm{k}) $ are imaginary and satisfy $ \hat{u}_{3}(z,\bm{k})=- \hat{u}_{3}(z,-\bm{k}) $. An analogous Fourier series representation exist for $ \eta_j\in C^{2,\delta}_{per,e}(\mathbb{R}^2) $
		
		For an operator $ F:\mathcal{X}\times \mathcal{Y}\to \mathcal{Z} $, where $ \mathcal{X}$,  $\mathcal{Y}$ and $\mathcal{Z}$ are Banach spaces, we denote the Fréchet derivative at $ (x_0,y_0)\in \mathcal{X}\times \mathcal{Y} $ by $DF[x_0,y_0]$. Moreover, by $ D_x F[x_0,y_0] $ and $D_y F[x_0,y_0]$ we denote the Fréchet derivatives of $ F[\cdot,y_0]:\mathcal{X}\to \mathcal{Z} $ and $ F[x_0,\cdot]:\mathcal{X}\to \mathcal{Z} $ at $ x_0\in \mathcal{X} $ and at $ y_0\in\mathcal{X} $, respectively. Finally, we note that
		\[ 
			DF[x_0,y_0](x,y)=D_xF[x_0,y_0](x)+D_yF[x_0,y_0](y)
		\]
		for $ F\in C^1(\mathcal{X}\times \mathcal{Y}, \mathcal{Z}) $.

		\section{Flattening}\label{sec:flat}
		To avoid unnecessarily complicated notation after we change variables, we change notation for the physical frame, and denote the coordinates and functions expressed in these coordinates with a bar, for example $\bar{\bm{u}}(\bar{\bm{x}})$.
		We flatten the domains $\Omega_j$ using a naive flattening given by
		\begin{align*}
			(\bar{x},\bar{y},\bar{z})=\Phi_j(x,y,z)&=\left(x,y,\left(1+\frac{\eta_j-\eta_{j-1}}{d_j-d_{j-1}}\right)z+\frac{d_j\eta_{j-1}-d_{j-1}\eta_j}{d_j-d_{j-1}}\right)\eqqcolon (x,y,\varphi_j(x,y,z)).
		\end{align*}
		Since we only are concerned with small amplitude waves this flattening transformation is sufficient, and it is conceptually easy to understand.
		$\Phi_j$ maps
		\[
			\Omega_j^0=\{(\bm{x}',z)\in \mathbb{R}^2\times \mathbb{R}: -d_{j-1}<z<d_j\}
		\]
		to $\Omega_j$ as long as the interfaces do not intersect, that is, if $\eta_{j-1}+d_{j-1}<\eta_{j}+d_j$ for $ j=1,\ldots,m $.
		The flattening transformation has Jacobian matrix
		\[ 
			\mathcal{J}_j\coloneqq D\Phi_j=\begin{pmatrix}
			1 & 0 & 0\\
			0 & 1 & 0\\
			\partial_x\varphi_j & \partial_y\varphi_j & \partial_z\varphi_j
			\end{pmatrix},
		\]
		with determinant
		\[
		J_j(\bm{x})\coloneqq \partial_z\varphi_j=1+\frac{\eta_j-\eta_{j-1}}{d_j-d_{j-1}}.
		\]
		For scalar functions, $ \bar{f}:\Omega_j\to \mathbb{R} $ we define $ f:\Omega^0_j\to \mathbb{R} $  through
		\[ 
		f=\bar{f}\circ \Phi,
		\]
		and for vectors fields, $ \bar{\bm{u}}^{(j)}:\Omega_j\to \mathbb{R}^3 $, we define $ \bm{u}^{(j)}:\Omega^0_j\to \mathbb{R}^3 $ through
		\[ 
		{\bm{u}}^{(j)}\circ{\Phi}^{-1}=J_j \mathcal{J}^{-1}_j\bar{\bm{u}}^{(j)},
		\]
		which corresponds to expressing the vector fields in terms of a position-dependent, and not necessarily orthonormal, basis. We also note that this transformation preserves the regularity and symmetry, that is, if $ \bm{\eta}\in \mathfrak{Y} $ then $ \bar{\bm{u}}^{(j)}\in \mathfrak{X}_j $ if and only if $ \bm{u}^{(j)}\in (C^{1,\delta}_{per,e}(\overline{\Omega_j^0}))^2\times C^{1,\delta}_{per,o}(\overline{\Omega_j^0})\eqqcolon \mathfrak{X}^0_j $. We also define $ \mathfrak{X}^0 $ analogously to $ \mathfrak{X} $.
		We use this flattening to transform our original free boundary problem to a problem in the fixed domains.
		\begin{proposition}\label{prop:flat_prob}
			Assume $\alpha_j(d_j-d_{j-1})\notin 2\pi\mathbb{Z}\setminus\{0\}$, then the equations
			\begin{align}
			\nabla\times \bm{v}^{(j)}-\alpha_j{\bm{v}}^{(j)}&=\nabla\times \bm{N}_j[\bm{v},\bm{\eta},\bm{\tau}]&&\inn \Omega_j^0,\label{eq:curl_flat}\\
			\nabla \cdot\bm{v}^{(j)}&=0&&\inn \Omega_j^0,\label{eq:div_flat}\\
			v_3^{(j)}&=U_1^{(j)}\partial_x\eta_{j-1}+U_2^{(j)}\partial_y\eta_{j-1} &&\onn z=d_{j-1},\label{eq:bdry_bottom_flat}\\
			v_3^{(j)}&=U_1^{(j)}\partial_x\eta_{j}+U_2^{(j)}\partial_y\eta_{j}&&\onn z=d_j,\label{eq:bdry_top_flat}\\
			\int_{\hat{\Omega}_j^0}v_i^{(j)}d\bm{x}&=0&& i=1,2\label{eq:int_flat}\\
			\intertext{for $ j=1,\ldots, m $ and}
			B_j[\bm{v},\bm{\eta},\bm{\tau}]+R_j[\bm{v},\bm{\eta},\bm{\tau}]&=0 && \onn z=d_j\label{eq:dyn_flat}
			\end{align}
			for $j=1,\ldots,n$, with $ \bm{\eta}\in \mathfrak{Y} $, and $\bm{v}=(\bm{v}^{(1)},\ldots,\bm{v}^{(m)})\in \mathfrak{X}^0$ are equivalent to the \crefrange{eq:curl}{eq:int_cond} with $ \bm{\eta}\in \mathfrak{Y} $ and $\bar{\bm{u}}\in \mathfrak{X}$, where
			\begin{align*}
			B_j[\bm{v},\bm{\eta},\bm{\tau}]&=\rho_j(U_1^{(j)} v^{(j)}_1+U_2^{(j)} v^{(j)}_2)
			-\rho_{j+1}(U_1^{(j+1)}v_1^{(j+1)}+U_2^{(j+1)}v_2^{(j+1)})\\
			&\qquad+(\rho_j-\rho_{j+1})g\eta_{j}-\sigma_{j}\Delta \eta_{j},
			\end{align*}
			and $ \bm{N}_j:\mathfrak{X}^0\times \mathfrak{Y}\times \mathfrak{Z}\to \mathfrak{X}_j $ and $ R_j:\mathfrak{X}^0\times \mathfrak{Y}\times \mathfrak{Z}\to \mathfrak{W}_j $ are some operators that satisfy
			 \begin{align*}
			 \bm{N}_j[0,0,\bm{\tau}]&=D_{\bm{v}}\bm{N}_j[0,0,\bm{\tau}](\bm{v})=D_{\bm{\eta}}\bm{N}_j[0,0,\bm{\tau}](\bm{\eta})=0,\\ R_j[0,0,\bm{\tau}]&=D_{\bm{v}}R_j[0,0,\bm{\tau}](\bm{v})=D_{\bm{\eta}}R_j[0,0,\bm{\tau}](\bm{\eta})=0.
			 \end{align*}
		\end{proposition}
		\begin{proof}
		By directly applying the flatting transformation to \crefrange{eq:curl}{eq:int_cond} turns the equations into
		\begin{align}
			\nabla\times \bm{u}^{(j)}-\alpha_j\bm{u}^{(j)}&=\nabla\times \bm{\mathcal{M}}_j[\bm{u},\bm{\eta}]&&\inn \Omega_j^0,\label{eq:curlnaiveflat}\\
			\nabla \cdot\bm{u}^{(j)}&=0&&\inn \Omega_j^0,\\
			u_3^{(j)}&=0&&\onn \partial\Omega_j^0,\\
			\int_{\hat{\Omega}_j^0}u_i^{(j)}d\bm{x}&=\int_{\hat{\Omega}_j}U_i^{(j)}[\bm{\tau}]d\bm{x}&& i=1,2\label{eq:intnaiveflat}\\
			\intertext{for $j=1,\ldots ,m$,}
			\mathcal{B}_j[\bm{u},\bm{\eta}]&=(\rho_j-\rho_{j+1})\frac{r^2}{2} &&\onn z=d_{j},\label{eq:dynbdrynaiveflat}
		\end{align}
		for $j=1,\ldots ,n$.
		where
		\[
			\bm{\mathcal{M}}_j[\bm{u},\bm{\eta}]=\bm{u}^{(j)}-\frac{1}{J_j}\mathcal{J}_j^T\mathcal{J}_j \bm{u}^{(j)},
		\]
		and
		\begin{align*}
			\mathcal{B}_j[\bm{u},\bm{\eta}]&=\rho_j\left(\frac{\bm{u}^{(j)}\cdot \mathcal{J}_j^T\mathcal{J}_j \bm{u}^{(j)}}{2 J_j^2}+g\eta_{j}\right)-\rho_{j+1}\left(\frac{\bm{u}^{(j+1)}\cdot \mathcal{J}_{j+1}^T\mathcal{J}_{j+1} \bm{u}^{(j+1)}}{2 J_{j+1}^2}+g\eta_{j}\right)\\
			&\qquad-\sigma_{j}{\nabla} \cdot \left(\frac{{\nabla} \eta_{j}}{\sqrt{1+\vert {\nabla} \eta_{j}\vert^2}}\right)
		\end{align*}
		These equations are very similar to \crefrange{eq:curl_flat}{eq:dyn_flat}, but there are two vital differences. The first is that $ D_\eta \bm{\mathcal{M}}_j[0,0](\bm{\eta})\neq 0 $ and the second is that the right hand side of \cref{eq:intnaiveflat} is nonzero. To rectify this we introduce $ \tilde{\bm{u}}^{(j)} $ and $ \tilde{\bm{U}}^{(j)} $, defined below, which will give equations with the desired properties for 
		\begin{equation}\label{eq:vdef} \bm{v}^{(j)}\coloneqq\bm{u}^{(j)}-\tilde{\bm{u}}^{(j)}-\tilde{\bm{U}}^{(j)}.
		\end{equation}
		The velocity field $ \tilde{\bm{u}}^{(j)} $ gives the right hand side of \cref{eq:curl_flat} the properties we want, but introduces the nonzero right hand sides of the boundary conditions in \crefrange{eq:bdry_bottom_flat}{eq:bdry_top_flat}, while $ \tilde{\bm{U}}^{(j)} $ makes the right hand side of \cref{eq:int_flat} zero. Specifically, $ \tilde{\bm{u}}^{(j)} $ is given by
		\begin{equation}\label{eq:tildeu_def}
		\tilde{\bm{u}}^{(j)}=
		\begin{pmatrix}
		(J_j-1)U_1^{(j)}+\alpha_j(\varphi_j-z)U_2^{(j)}\\
		(J_j-1)U_2^{(j)}-\alpha_j(\varphi_j-z)U_1^{(j)}\\
		-\partial_x\varphi_j U_1^{(j)}-\partial_y\varphi_j U_2^{(j)}
		\end{pmatrix},
		\end{equation}
		and $\tilde{\bm{U}}^{(j)}=\tilde{\bm{U}}^{(j)}[\bm{\tau},\bm{\eta}]$ is a laminar flow such that $\tilde{\bm{U}}^{(j)}[\bm{\tau},0]=\bm{U}^{(j)}[\bm{\tau}]$, that is, for flat interfaces it coincides with the solution from \Cref{sec:triv_sol}.
		Moreover, $\tilde{\bm{U}}^{(j)}[\bm{\tau},\bm{\eta}]$ is chosen in such a way that
		\[ 
		\int_{\hat{\Omega}_j^0}v_i^{(j)}d\bm{x}=0,\qquad i=1,2.
		\]
		To show that this choice is possible we begin by computing
		\begin{align*} 
			\int_{d_{j-1}}^{d_j} \tilde{u}_1^{(j)} dz&=\int_{d_{j-1}}^{d_j}\left(\alpha_j(\varphi_j-z)U_2^{(j)}+(J_j-1)U_1^{(j)}\right)dz\\
			&=\eta _{j}U_{1}\left( d_{j}\right) -\eta _{j-1}U_{1}\left( d_{j-1}\right) -\int_{d_{j-1}}^{d_{j}}(J_j-1)U_{1}\left( z\right) dz+\int_{d_{j-1}}^{d_{j}}(J_j-1)U_{1}\left( z\right) dz\\
			&=\eta _{j}U_{1}\left( d_{j}\right) -\eta _{j-1}U_{1}\left( d_{j-1}\right) 
		\end{align*}
		and similarly
		\[ 
			\int_{d_{j-1}}^{d_j} \tilde{u}_2^{(j)} dz=\eta _{j}U_{2}\left( d_{j}\right) -\eta _{j-1}U_{2}\left( d_{j-1}\right).
		\]
		Since
		\[ 
			\int_{\hat{\Omega}_j^0}v_i^{(j)}d\bm{x}=\underbrace{\int_{\hat{\Omega}_j} U_i^{(j)} d\bm{x}-\int_{\hat{\Omega}_j^0}\tilde{u}^{(j)}_i d\bm{x}-\int_{\hat{\Omega}_j^0}U^{(j)}_i d\bm{x}}_{\mathcal{I}_i^{(j)}[\bm{\tau},\bm{\eta}]}
			-\int_{\hat{\Omega}_j^0}(\tilde{U}_i-U^{(j)}_i) d\bm{x},\qquad i=1,2,
		\]
		we find
		\begin{align*}
			\mathcal{I}_1^{(j)}[\bm{\tau},\bm{\eta}]
			&=\int_{\Gamma} \left(\int_{d_{j-1}+\eta_{j-1}}^{d_j+\eta_j}U^{(j)}_1 dz-\int_{d_{j-1}}^{d_j} (\tilde{u}^{(j)}_1+U^{(j)}_1 )dz \right)d\bm{x}'\\
			&=-\int_{\Gamma} \left(\frac{U_2^{(j)}(d_j+\eta_j)-U_2^{(j)}(d_{j-1}+\eta_{j-1})}{\alpha_{j}}\right.\\
			&\qquad\qquad-\left.\frac{U_2^{(j)}(d_j)-U_2^{(j)}(d_{j-1})}{\alpha_{j}} +(\eta _{j}U_{1}^{(j)}\left( d_{j}\right) -\eta _{j-1}U_{1}^{(j)}\left( d_{j-1}\right) )\right)d\bm{x}'\\
			&=-\frac{1}{\alpha_{j}}\int_{\Gamma} \left((U_2^{(j)}(d_j+\eta_j)-U_2^{(j)}(d_j)-\eta_j\partial_z U_2^{(j)}(d_j))\right.\\
			&\qquad\qquad-\left.(U_2^{(j)}(d_{j-1}+\eta_{j-1})-U_2^{(j)}(d_{j-1})-\eta_{j-1}\partial_zU_2^{(j)}(d_{j-1}))\right) d\bm{x}'\\
			&=\frac{1}{\alpha_{j}}\int_{\Gamma} \sum_{n=2}^\infty \left(\eta_{j-1}^n\frac{\partial_z^n U_2^{(j)}(d_{j-1})}{n!}-\eta_{j}^n\frac{\partial_z^n U_2^{(j)}(d_{j})}{n!}\right)d\bm{x}'
		\end{align*}
		and similarly
		\[ 
			\mathcal{I}_2^{(j)}[\bm{\tau},\bm{\eta}]=\frac{1}{\alpha_{j}}\int_{\Gamma} \sum_{n=2}^\infty \left(\eta_{j}^n\frac{\partial_z^n U_1^{(j)}(d_{j})}{n!}-\eta_{j-1}^n\frac{\partial_z^n U_1^{(j)}(d_{j-1})}{n!}\right)d\bm{x}'.
		\]
		By setting $ \tilde{\bm{U}}^{(j)}[\bm{\tau},\bm{\eta}]-\bm{U}[\bm{\tau},0]=c_1^{(j)}[\bm{\tau},\bm{\eta}] \bm{V}_1+c_2^{(j)}[\bm{\tau},\bm{\eta}]\bm{V}_2 $ and integrating we obtain that we need $c_1^{(j)}[\bm{\tau},\bm{\eta}]$ and $c_2^{(j)}[\bm{\tau},\bm{\eta}]$ to satisfy
		\[ 
			\frac{\vert \Gamma\vert}{\alpha_{j}}\begin{pmatrix}
				\sin(\alpha_jd_j)-\sin(\alpha_jd_{j-1}) & -\cos(\alpha_jd_j)+\cos(\alpha_jd_{j-1})\\
				\cos(\alpha_jd_j)-\cos(\alpha_jd_{j-1})& \sin(\alpha_jd_j)-\sin(\alpha_jd_{j-1})
			\end{pmatrix}
			\begin{pmatrix}
			c_1^{(j)}[\bm{\tau},\bm{\eta}]\\
			c_2^{(j)}[\bm{\tau},\bm{\eta}]
			\end{pmatrix}=
			\begin{pmatrix}
			\mathcal{I}_1[\bm{\tau},\bm{\eta}]\\
			\mathcal{I}_2[\bm{\tau},\bm{\eta}]
			\end{pmatrix}.
		\]
		If $ \alpha_j(d_j-d_{j-1})\notin 2\pi\mathbb{Z} $ these equations are solvable and we get
		\[ 
			\begin{pmatrix}
			c_1^{(j)}[\bm{\tau},\bm{\eta}]\\
			c_2^{(j)}[\bm{\tau},\bm{\eta}]
			\end{pmatrix}=\frac{\alpha_j}{\vert \Gamma\vert }\begin{pmatrix}
			\frac{\sin(\alpha_jd_j)-\sin(\alpha_jd_{j-1}) }{2-2\cos(\alpha_j(d_j-d_{j-1}))}& \frac{\cos(\alpha_jd_j)-\cos(\alpha_jd_{j-1}) }{2-2\cos(\alpha_j(d_j-d_{j-1}))}\\
			\frac{-\cos(\alpha_jd_j)+\cos(\alpha_jd_{j-1}) }{2-2\cos(\alpha_j(d_j-d_{j-1}))}& \frac{\sin(\alpha_jd_j)-\sin(\alpha_jd_{j-1}) }{2-2\cos(\alpha_j(d_j-d_{j-1}))}
			\end{pmatrix}
			\begin{pmatrix}
			\mathcal{I}_1[\bm{\tau},\bm{\eta}]\\
			\mathcal{I}_2[\bm{\tau},\bm{\eta}]
			\end{pmatrix}
		\]
		We note that $\tilde{\bm{U}}^{(j)}[\bm{\tau},\bm{\eta}]-\bm{U}[\bm{\tau},0]$ is of quadratic order with respect to $\bm{\eta}$. Moreover, for $ \alpha_j=0 $ obvious modifications to the calculations give $ c_1^{(j)}=c_2^{(j)}=0 $, which coincide with the limit $ \alpha_j\to 0 $ in the formula above. In fact, setting $ c_1^{(j)}[\bm{\tau},\bm{\eta}]=c_2^{(j)}[\bm{\tau},\bm{\eta}]=0 $ for $ \alpha_j=0 $ gives an analytic function (with respect to $ \alpha_j $) around $ 0 $. Thus we can find the desired $ \tilde{\bm{U}}^{(j)} $ for $ \alpha_j(d_j-d_{j-1})\notin 2\pi\mathbb{Z}\setminus\{0\} $.
		
		We move on to check what happens with the other equations.
		First note that
		\begin{align}
		\nabla\times \tilde{\bm{u}}^{(j)}-\alpha_j\tilde{\bm{u}}^{(j)}&=\nabla\times \begin{pmatrix}
		(J_j-1)\bm{U}_1^{(j)}\\
		(J_j-1)\bm{U}_2^{(j)}\\
		-\partial_x\varphi_j\bm{U}_1^{(j)} -\partial_y\varphi_j\bm{U}_2^{(j)}
		\end{pmatrix}&& \inn \Omega^0_j,\label{eq:tildeu_prop1}\\
		\nabla\cdot \bm{\tilde{u}}^{(j)}&=0&& \inn \Omega^0_j,\label{eq:tildeu_prop2}\\
		v_3^{(j)}&=U_1^{(j)}\partial_x\eta_{j-1}+U_2^{(j)}\partial_y\eta_{j-1} &&\onn z=d_{j-1},\label{eq:tildeu_prop3}\\
		v_3^{(j)}&=U_1^{(j)}\partial_x\eta_{j}+U_2^{(j)}\partial_y\eta_{j}&&\onn z=d_j.\label{eq:tildeu_prop4}
		\end{align}
		Then we define $\widetilde{\bm{\mathcal{M}}}_j[\bm{v},\bm{\eta},\bm{\tau}]=\bm{\mathcal{M}}_j[\bm{v}+\tilde{\bm{u}}[\bm{\eta},\bm{\tau}]+\tilde{\bm{U}}[\bm{\eta},\bm{\tau}],\bm{\eta}] $ and $\tilde{\mathcal{B}}_j[\bm{v},\bm{\eta},\bm{\tau}]=\mathcal{B}_j[\bm{v}+\tilde{\bm{u}}[\bm{\eta},\bm{\tau}]+\tilde{\bm{U}}[\bm{\eta},\bm{\tau}],\bm{\eta}] $ and find that
		\begin{align*} 
			\bm{M}_j[\bm{v},\bm{\eta},\bm{\tau}]\coloneqq& D_{\bm{v}}\widetilde{\bm{\mathcal{M}}}_j[0,0,\bm{\tau}](\bm{v})+D_{\bm{\eta}}\widetilde{\bm{\mathcal{M}}}_j[0,0,\bm{\tau}](\bm{\eta})=
			\begin{pmatrix}
				(J_j-1)\bm{U}_1^{(j)}\\
				(J_j-1)\bm{U}_2^{(j)}\\
				-\partial_x\varphi_j\bm{U}_1^{(j)} -\partial_y\varphi_j\bm{U}_2^{(j)}
			\end{pmatrix},
		\end{align*}
		and
		\begin{align*}
			B_j[\bm{v},\bm{\eta},\bm{\tau}]\coloneqq D_{\bm{v}}\tilde{\mathcal{B}}_j[0,0,\bm{\tau}](\bm{v})+D_{\bm{\eta}}\tilde{\mathcal{B}}_j[0,0,\bm{\tau}](\bm{\eta})&=\rho_j(U_1^{(j)} v^{(j)}_1+U_2^{(j)} v^{(j)}_2)\\
			&\qquad-\rho_{j+1}(U_1^{(j+1)}v_1^{(j+1)}+U_2^{(j+1)}v_2^{(j+1)})\\
			&\qquad+(\rho_j-\rho_{j+1})g\eta_{j}-\sigma_{j}\Delta \eta_{j}.
		\end{align*}
		Combining these two equations above with \crefrange{eq:vdef}{eq:tildeu_prop4} we see that \crefrange{eq:curlnaiveflat}{eq:dynbdrynaiveflat} turn into \crefrange{eq:curl_flat}{eq:dyn_flat} with $ \bm{v}$ as a variable instead of $ \bm{u} $, where, $\bm{N}_j[\bm{v},\bm{\eta},\bm{\tau}]=\widetilde{\bm{\mathcal{M}}}_j[\bm{v},\bm{\eta},\bm{\tau}]-\bm{M}_j[\bm{v},\bm{\eta},\bm{\tau}]$ and $R_j[\bm{v},\bm{\eta},\bm{\tau}]=\tilde{\mathcal{B}}_j[\bm{v},\bm{\eta},\bm{\tau}]-B_j[\bm{v},\bm{\eta},\bm{\tau}]$. From these definitions it also directly follows that \begin{align*}
		\bm{N}_j[0,0,\bm{\tau}]&=D_{\bm{v}}\bm{N}_j[0,0,\bm{\tau}](\bm{v})=D_{\bm{\eta}}\bm{N}_j[0,0,\bm{\tau}](\bm{\eta})=0\\ R_j[0,0,\bm{\tau}]&=D_{\bm{v}}R_j[0,0,\bm{\tau}](\bm{v})=D_{\bm{\eta}}R_j[0,0,\bm{\tau}](\bm{\eta})=0.
		\end{align*}
		\end{proof}	
		
	In light of this proposition we focus our attention to \crefrange{eq:curl_flat}{eq:dyn_flat} in the remainder of the paper.
	\subsection{Reduction to the interfaces}
	We define the spaces
	\begin{align*} 
		\mathbb{X}_j&=\left\{\bm{u}\in \mathfrak{X}_j:\nabla\cdot\bm{u}=0,\int_{\hat{\Omega}^0_j}u_id\bm{x}=0,\; i=1,2\right\}\\
		\mathbb{Y}_j&=\left\{ (\bm{v},f,g)\in ((C^{0,\delta}_{per,e}(\overline{\Omega^0_j}))^2\times C^{0,\delta}_{per,o}(\overline{\Omega^0_j}))\times C^{1,\delta}_{per,o}(\mathbb{R}^2)\times C^{1,\delta}_{per,o}(\mathbb{R}^2):\nabla \cdot \bm{v}=0\right\}
	\end{align*}
	and the operator $ \mathcal{C}_{\alpha_j,j}:\mathbb{X}_j\to\mathbb{Y}_j $ by
	\[ 
		\mathcal{C}_{\alpha_j,j}[\bm{u}]=(\nabla\times \bm{u}-\alpha_j{\bm{u}},u_3\vert_{z=d_j},u_3\vert_{z=d_{j-1}}).
	\]
	For a given $ j=1,\cdots,m $ the \crefrange{eq:curl_flat}{eq:int_flat} are equivalent to
	\[ 
		\mathcal{C}_{\alpha_j,j}[\bm{v}]=(\nabla\times \bm{N}_j[\bm{v},\bm{\eta},\bm{\tau}],U_1^{(j)}\partial_x\eta_{j}+U_2^{(j)}\partial_y\eta_{j},U_1^{(j)}\partial_x\eta_{j-1}+U_2^{(j)}\partial_y\eta_{j-1}).
	\]
		To find a unique $ \bm{v} $ solving \crefrange{eq:curl_flat}{eq:int_flat} if $ \bm{\eta} $ and $ \bm{\tau} $ are given we need to impose the non-resonance condition
		\begin{equation}\label{eq:no-resonance}
			\sqrt{\alpha_j^2-\vert\bm{k}^2\vert}\notin 	\frac{\pi\mathbb{Z}_+}{d_j-d_{j-1}}\text{ for all }\bm{k}\in\Lambda'\text{ such that }\vert \bm{k}\vert<\vert \alpha_j\vert\text{ and all }j=1,\ldots,m.
		\end{equation}
		Doing so allows us to substitute this solution, $ \bm{v}[\bm{\eta},\bm{\tau}] $, into \cref{eq:dyn_flat}, which reduces the problem to the interfaces. We also note that \cref{eq:no-resonance} is a stricter condition than the assumption in \Cref{prop:flat_prob}; $ \bm{k}=0 $ for $ \alpha_j\neq 0 $ in \cref{eq:no-resonance} implies that the assumption in \Cref{prop:flat_prob} is satisfied. To solve \crefrange{eq:curl_flat}{eq:int_flat} we begin by proving that $\mathcal{C}_{\alpha_j,j}$ has the following properties.
		\begin{lemma}\label{lemma:iso}\leavevmode
			\begin{itemize}
				\item[(i)] $ \mathcal{C}_{0,j}:\mathbb{X}_j\to\mathbb{Y}_j $ is an isomorphism.
				\item[(ii)]$ \mathcal{C}_{\alpha_j,j}:\mathbb{X}_j\to\mathbb{Y}_j $ is a Fredholm operator with index 0.
				\item[(iii)] $ \mathcal{C}_{\alpha_j,j}:\mathbb{X}_j\to\mathbb{Y}_j $ is an isomorphism if and only if the non-resonance condition in \cref{eq:no-resonance} is satisfied.
			\end{itemize}
		\end{lemma}
		\begin{proof} Clearly $ \mathcal{C}_{\alpha_j,j}:\mathbb{X}_j\to \mathbb{Y}_j $ is a bounded linear map. To prove $ (i) $ we begin by proving that $ \mathcal{C}_{0,j} $ is injective. If
		\begin{align*}
		\nabla\times \bm{u}&=0 &&\inn {\Omega}^0_j\\
		\nabla\cdot \bm{u}&=0&&\inn {\Omega}^0_j,\\
		u_3&=0&&\onn z=d_j,\\
		u_3&=0 && \onn z={d_{j-1}},
		\end{align*}
		then $ \Delta\bm{u}=0 $ so $ u_3=0 $, which means that
		\begin{align*}
			\partial_{z}u_1=\partial_{z}u_2&=0&& \inn \Omega^0_j.
		\end{align*}
		By continuity $ u_1 $ and $ u_2 $ satisfy Laplace equation with Neumann boundary conditions on $ z=d_j $ and $ z=d_{j-1} $. Together with the fact that $ u_1 $ and $ u_2 $ are periodic and satisfy the integral conditions
		\[ 
			\int_{\hat{\Omega}^0_j}u_i^{(j)}d\bm{x}=0,\qquad i=1,2,
		\]
		we obtain $ u_1=u_2=0 $. Hence $ \mathcal{C}_{0,j} $ is injective. To show that $ \mathcal{C}_{0,j} $ is surjective pick a general element $ (\bm{v},f,g)\in \mathbb{Y}_j $. We introduce $ \bm{A} $ satisfying
		\begin{align*}
			\Delta \bm{A}&=\bm{v}&&\inn \Omega_j^0,\\
			A_1=A_2=\partial_{z}A_3&=0 &&\onn  z=d_j,\\
			A_1=A_2=\partial_{z}A_3&=0 &&\onn  z=d_{j-1},\\
			\intertext{and $ \phi $ satisfying}
			\Delta \phi&=0&&\inn \Omega_j^0,\\
			\partial_{z}\phi=&=f &&\onn  z=d_j,\\
			\partial_{z}\phi=&=g &&\onn  z=d_{j-1}.
		\end{align*}
		By standard elliptic theory we can solve both these problems in the appropriate function spaces and it is not hard to check that $ \bm{u}=-\nabla\times \bm{A}+\nabla \phi $ satisfies
		\[ 
			\mathcal{C}_{0,j} \bm{u}=(\bm{v},f,g).
		\]
		For property $ (ii) $ we note that $ \mathcal{C}_{0,j}^{-1}\mathcal{C}_{\alpha_j,j}:\mathbb{X}_j\to \mathbb{X}_j $ is of the from `identity plus a compact operator'. Thus $ \mathcal{C}_{0,j}^{-1}\mathcal{C}_{\alpha_j,j}:\mathbb{X}_j\to \mathbb{X}_j $ is Fredholm with index $ 0 $ and it follows that so is $\mathcal{C}_{\alpha_j,j}:\mathbb{X}_j\to \mathbb{Y}_j $.
		
		Part $ (iii) $ follows from part $ (ii) $ and the fact that $ \mathcal{C}_{\alpha_j,j} $ is injective if and only if \cref{eq:no-resonance} is satisfied. Indeed, if we consider 
		\[ 
			\mathcal{C}_{\alpha_j,j}\bm{u}=(\bm{0},0,0)
		\]
		then the Fourier coefficients of $ \bm{u}(\bm{x}',z)=\sum_{\bm{k}\in\Lambda'} (\hat{u}_{1}(z,\bm{k}),\hat{u}_{2}(z,\bm{k}),\hat{u}_{3}(z,\bm{k}))e^{i\bm{k}\cdot \bm{x}'} $ must satisfy
		\begin{align} 
			\partial_z^2\hat{u}_{i}(z,\bm{k})+(\alpha_j^2-\vert\bm{k}\vert ^2)\hat{u}_{i}(z,\bm{k})&=0&&\inn \Omega_j^0,\label{eq:fourier_coeff}\\
			\hat{u}_{3}(z,\bm{k})&=0 &&\onn  z=d_j,\label{eq:fourier_coeffbdry1}\\
			\hat{u}_{3}(z,\bm{k})&=0&&\onn  z=d_{j-1}.\label{eq:fourier_coeffbdry2}
		\end{align}
		Under the condition in \cref{eq:no-resonance} this implies $ \hat{u}_{3}(z,\bm{k})\equiv 0 $ in $ \Omega_j^0 $ for all $ \bm{k}\in \Lambda' $. With this in mind we find that
		\begin{align*}
			(\bm{k}\cdot \bm{e}_1)\hat{u}_{1}(z,\bm{k})+(\bm{k}\cdot \bm{e}_2)\hat{u}_{2}(z,\bm{k})&=0&&\inn \Omega_j^0,\\
			-(\bm{k}\cdot \bm{e}_2)\hat{u}_{1}(z,\bm{k})+(\bm{k}\cdot \bm{e}_1)\hat{u}_{2}(z,\bm{k})&=0&&\inn \Omega_j^0,
		\end{align*}
		which implies $ \hat{u}_{1}(z,\bm{k})\equiv \hat{u}_{2}(z,\bm{k})\equiv 0 $ in $ \Omega_j^0 $ for all $ \bm{k}\in\Lambda' $, $ \bm{k}\neq 0 $. Moreover, \cref{eq:fourier_coeff}, the condition in \cref{eq:no-resonance} and the integral conditions $ \int_{\hat{\Omega}^0_j}u_i d\bm{x}=0$, $ i=1,2, $ give $ \hat{u}_{1,\bm{0}}\equiv \hat{u}_{2,\bm{0}}\equiv 0 $ in $ \Omega_j^0 $. Hence, we have shown $ \mathcal{C}_{\alpha_j,j} $ is injective under the condition in \cref{eq:no-resonance}. On the other hand, if \cref{eq:no-resonance} is not satisfied we can a find non-trivial solution $ \hat{u}_{3}(z,\bm{k}) $ to \crefrange{eq:fourier_coeff}{eq:fourier_coeffbdry2} for the $ \bm{k} $ that breaks the condition in \cref{eq:no-resonance}. This in turn gives a non-trivial element in the kernel of $ \mathcal{C}_{\alpha_j,j} $.
		\end{proof}
	
		\begin{proposition}\label{prop:reduction_to_bdry} 
			Assume the non-resonance condition in \eqref{eq:no-resonance} holds, then for any $\bm{\eta}\in \mathfrak{Y}$, $\bm{\tau}\in \mathfrak{Z}$, and $j\in\{1,\ldots, m\}$ there exists a unique solution $\bm{v}^{(j)}\in \mathfrak{X}_j$ to the problem given by \crefrange{eq:curl_flat}{eq:int_flat} provided that $\Vert \eta_{j-1}\Vert_{\mathfrak{Y}_{j-1}}\ll 1$ and $\Vert \eta_j\Vert_{\mathfrak{Y}_{j}}\ll 1$. Furthermore, $\bm{v}^{(j)}$ depends analytically on $\bm{\tau}$ and $\bm{\eta}$. If $\bm{\eta}$ is constant in the direction of $\bm{\lambda}_i$ then so is $\bm{v}^{(j)}$.
		\end{proposition}
		\begin{proof} 
			We can write \crefrange{eq:curl_flat}{eq:int_flat} as
			\[ 
				\bm{G}[\bm{v},\bm{\eta},\bm{\tau}]=0,
			\]
			where $ \bm{G}:\mathbb{X}\to\mathbb{Y} $ is given by
			\begin{align*}
				G_j[\bm{v},\bm{\eta},\bm{\tau}]&=(\nabla\times \bm{v}^{(j)}-\alpha_j{\bm{v}}^{(j)}-\nabla\times \bm{N}_j[\bm{v},\bm{\eta},\bm{\tau}],\\
				&\qquad [v_3^{(j)}-U_1^{(j)}\partial_x\eta_{j}-U_2^{(j)}\partial_y\eta_{j}]\vert_{z=d_j},[v_3^{(j)}-U_1^{(j)}\partial_x\eta_{j-1}-U_2^{(j)}\partial_y\eta_{j-1}]\vert_{z=d_{j-1}})
			\end{align*}
			for the spaces $ \mathbb{X}=\mathbb{X}_1\times\ldots\times\mathbb{X}_m $ and $ \mathbb{Y}=\mathbb{Y}_1\times\ldots\times\mathbb{Y}_m $. The Fréchet derivative at $ \bm{\eta}=0 $ is given by
			\[ 
				D_{\bm{v}}\bm{G}[\bm{u},\bm{0},\bm{\tau}](\bm{v})=\bm{\mathcal{C}}_{\bm{\alpha}}\bm{v}=(\mathcal{C}_{\alpha_1,1}\bm{v}^{(1)},\ldots,\mathcal{C}_{\alpha_m,m}\bm{v}^{(m)}).
			\]
			By \cref{lemma:iso} every component of $ \bm{\mathcal{C}}_{\bm{\alpha}} $ is an isomorphism, whence so is $  \bm{\mathcal{C}}_{\bm{\alpha}} $. This means we can apply the analytic implicit function theorem \cite[Theorem 4.5.4]{Buffoni_2003} to obtain the result.
			
			The analysis can be repeated in function spaces where the elements are constant in the direction of $\bm{\lambda}_i$. This gives the result that if $\bm{\eta}$ is constant in the direction of $\bm{\lambda}_i$ then so is $\bm{v}^{(j)}$.
		\end{proof}
		With the solution $\bm{v}[\bm{\eta},\bm{\tau}]$ from this proposition we can write \crefrange{eq:curl_flat}{eq:dyn_flat} as the single equation
		\begin{equation}\label{eq:SingleSurf}
			\bm{F}[\bm{\eta},\bm{\tau}]=0,
		\end{equation}
		where $ \bm{F}:\mathfrak{Y}\times \mathfrak{Z}\mapsto \mathfrak{W}$ is the operator defined through
		\begin{align*} 
			F_j[\bm{\eta},\bm{\tau}]=\left[B_j[\bm{v}[\bm{\eta},\bm{\tau}],\bm{\eta},\bm{\tau}]+R_j[\bm{v}[\bm{\eta},\bm{\tau}],\bm{\eta},\bm{\tau}]\right]_{z=d_j}.
		\end{align*}
		
	\section{The linearised problem}\label{sec:lin_prob}
		The aim is to apply the bifurcation theorem in \Cref{sec:AppendixBif} to find nontrivial solutions to \Cref{eq:SingleSurf}. To this end we have to study the linearised problem to obtain the Fréchet derivatives $ D_{\bm{\eta}} \bm{F}[0,\bm{\tau}]$ and $ D_{\bm{\tau}}D_{\bm{\eta}}\bm{F}[0,\bm{\tau}]$.
		If we set $\bm{W}^{(j)}=\bm{W}^{(j)}[\bm{\eta},\bm{\tau}]\coloneqq D_{\bm{\eta}}\bm{v}^{(j)}[0,\bm{\tau}](\bm{\eta})$, $ j=1,\cdots,{m} $, then
		\begin{equation}\label{eq:surf_prob_lin}
		\begin{aligned}
		D_{\bm{\eta}}F_j[0,\bm{\tau}](\bm{\eta})&=B_j[\bm{W},\bm{\eta},\bm{\tau}]\vert_{z=d_j}\\
		&=\bigg[\rho_j(U_1^{(j)} W^{(j)}_1+U_2^{(j)} W^{(j)}_2+g\eta_{j})\\
			&\qquad-\rho_{j+1}(U_1^{(j+1)}W_1^{(j+1)}+U_2^{(j+1)}W_2^{(j+1)}+g\eta_{j})
			-\sigma_{j}\Delta \eta_{j}\bigg]_{z=d_j}.
		\end{aligned}
		\end{equation}
		To show that this satisfies the properties required to apply \Cref{thm:bif} we will rewrite it in a more workable form summarized in \cref{lemma:Aproperties} at the end of this section. We cannot immediately give the result since we need to find $ \bm{W}^{(j)} $ in terms of $ \bm{\eta} $ and $ \bm{\tau} $ first. $ \bm{W}^{(j)} $ solves the linearised version of \crefrange{eq:curl_flat}{eq:int_flat} given by
		\begin{align}
			\nabla\times \bm{W}^{(j)}-\alpha_j{\bm{W}}^{(j)}&=0&&\inn \Omega_j,\label{eq:curl_lin}\\
			\nabla \cdot\bm{W}^{(j)}&=0&&\inn \Omega_j^0,\label{eq:div_lin}\\
			W_3^{(j)}&=U_1^{(j)}\partial_x\eta_{j-1}+U_2^{(j)}\partial_y\eta_{j-1} &&\onn z=d_{j-1},\label{eq:bdry_below_lin}\\
			W_3^{(j)}&=U_1^{(j)}\partial_x\eta_{j}+U_2^{(j)}\partial_y\eta_{j}&&\onn z=d_j,\label{eq:bdry_above_lin}\\
			\int_{\hat{\Omega}_j^0}W_i^{(j)}d\bm{x}&=0&& i=1,2.\label{eq:int_lin}
			\end{align}
			for $j=1,\ldots,m$. We express $ \bm{\eta} $ in terms of its Fourier series
			\[ 
				\bm{\eta}=\sum_{k\in \Lambda'} \hat{\bm\eta}(\bm{k}) e^{i\bm{k}\cdot\bm{x}'},
			\]
			where $ \hat{\bm{\eta}}(k)=(\hat{\eta}_1(k),\cdots,\hat{\eta}_n(k)) $. Since the problem is linear we can consider every Fourier mode of $ \bm{\eta} $ separately.
			Thus assume $ \bm{\eta}=\hat{\bm{\eta}}(\bm{k})e^{i\bm{k}\cdot\bm{x}'} $ for some $ \bm{k}\in\Lambda' $ and recall the polar form $\bm{k}=k(\cos(\gamma),\sin(\gamma))$. Then the solution will be of the form
				\[ 
					\bm{W}^{(j)}(x,y,z)=(W^{(j)}_1(z),W^{(j)}_2(z),W^{(j)}_3(z))e^{i\bm{k}\cdot\bm{x}'}
				\]
			Combining \cref{eq:curl_lin,eq:div_lin} gives the equation
			\[ 
				(W^{(j)}_3)''(z)-(\vert\bm{k}\vert^2-\alpha_j^2)W^{(j)}_3(z)=0
			\]
			for $ W^{(j)}_3 (z)$. This equation together with the boundary conditions
			\begin{align*}
			W^{(j)}_3(d_{j-1})&=i(U_1^{(j)}(d_{j-1})k\cos(\gamma)\hat{\eta}_{j-1}(\bm{k})+U_2^{(j)}(d_{j-1})k\sin(\gamma)\hat{\eta}_{j-1}(\bm{k})),\\
			W^{(j)}_3(d_{j})&=i(U_1^{(j)}(d_{j})k\cos(\gamma)\hat{\eta}_{j}(\bm{k})+U_2^{(j)}(d_{j})k\sin(\gamma)\hat{\eta}_{j}(\bm{k})),
			\end{align*}
			obtained from \cref{eq:bdry_below_lin,eq:bdry_above_lin}, has a unique solution when the non-resonance condition \eqref{eq:no-resonance} is satisfied. When $ \bm{k}=0 $ we obtain $ W^{(j)}_3=0 $. Moreover, in this case we also get $ W^{(j)}_1(z)=W^{(j)}_2(z)=0 $ due to the integral conditions in \cref{eq:int_lin} together with the non-resonance condition. Substituting this into \cref{eq:surf_prob_lin} gives
			\begin{equation}\label{eq:lin_prob_k0}
				F_j[0,\bm{\tau}](\hat{\bm{\eta}}(0))=(\rho_j-\rho_{j+1})g\hat{\eta}_j(0)
			\end{equation}
		
		To express the solution for $ \bm{k}\neq 0 $ we introduce $\phi_j(z,\bm{k})$ and $ \psi_j(z,\bm{k}) $ which solve
		\[
		\begin{aligned}
			\phi_j''(z,\bm{k})-(\vert \bm{k}\vert^2-\alpha_j^2)\phi_j(z,\bm{k})&=0,\\
			\phi_j(d_{j-1},\bm{k})&=1,\\
			\phi_j(d_j,\bm{k})&=0,
		\end{aligned}
		\qquad
		\begin{aligned}
			\psi_j''(z,\bm{k})-(\vert \bm{k}\vert^2-\alpha_j^2)\psi_j(z,\bm{k})&=0,\\
			\psi_j(d_{j-1},\bm{k})&=0,\\
			\psi_j(d_j,\bm{k})&=1,
		\end{aligned}
		\]
		respectively, where prime denotes derivative with respect to $ z $.
		Explicitly $ \psi_j(z,\bm{k}) $ is given by
		\[
			\psi_j(z,\bm{k})=\left\{
			\begin{aligned}
				&\frac{\sinh(\sqrt{(\vert \bm{k}\vert^2-\alpha_j^2)}(z-d_{j-1}))}{\sinh(\sqrt{(\vert \bm{k}\vert^2-\alpha_j^2)}(d_{j}-d_{j-1}))} &&\vert k\vert >\vert \alpha_j\vert,\\
				&\frac{z-d_{j-1}}{d_{j}-d_{j-1}}&&\vert \bm{k}\vert =\vert \alpha_j\vert \\
				&\frac{\sin(\sqrt{(\vert \bm{k}\vert^2-\alpha_j^2)}(z-d_{j-1}))}{\sin(\sqrt{(\vert \bm{k}\vert^2-\alpha_j^2)}(d_{j}-d_{j-1}))} &&\vert k\vert <\vert \alpha_j\vert,
			\end{aligned}\right.
		\]
		and from the equation it can easily be seen that we can obtain $ \phi_j(z,\bm{k}) $ by interchanging $ d_j $ and $ d_{j-1} $ in the expression for $ \psi_j(z,\bm{k}) $. 
		In terms of these functions we get
		\[
			W_3^{(j)}(z)=irk(\hat{\eta}_{j-1}\beta_{j-1}(\bm{k})\phi_j(z,\bm{k})+\hat{\eta}_{j}\beta_j(\bm{k})\psi_j(z,\bm{k})),
		\]
		where $ \beta_j(\bm{k})\coloneqq \cos(\theta_j-\alpha_jd_j-\gamma) $ satisfying
		\[
			(U_1^{(j)}(d_j),U_2^{(j)}(d_j))\cdot \bm{k}=rk\beta_j(\bm{k}).
		\] 
		We have also used the fact that our choice of $\bm{\theta}$ means that
		\[
		(U_1^{(j)}(d_{j-1}),U_2^{(j)}(d_{j-1}))\cdot \bm{k}=(U_1^{(j-1)}(d_{j-1}),U_2^{(j-1)}(d_{j-1}))\cdot \bm{k}=rk\beta_{j-1}(\bm{k})
		\]
		For future reference we also define $ \beta_j^\perp(\bm{k})\coloneqq \sin(\theta_j-\alpha_jd_j-\gamma) $, which satisfies
		\[ 
			(U_1^{(j)}(d_j),U_2^{(j)}(d_j))\cdot \bm{k}^\perp=rk\beta_j^\perp(\bm{k}),
		\]
		where $ \bm{k}^\perp=k(-\sin(\gamma),\cos(\gamma)) $.
		Using the expression for $ W_3^{(j)}(z) $ we find
		\begin{align*}
			W_1^{(j)}(z)&=-\frac{i}{\vert\bm{k}\vert^2}(-k\cos(\gamma)\partial_z W_3^{(j)}(z)-k\sin(\gamma)\alpha_j W_3^{(j)}(z)),\\
			W_2^{(j)}(z)&=-\frac{i}{\vert\bm{k}\vert^2}(-k\sin(\gamma)\partial_z W_3^{(j)}(z)+k\cos(\gamma)\alpha_j W_3^{(j)}(z)).
		\end{align*}
		Substituting these expressions for $\bm{W}$ into \cref{eq:surf_prob_lin} gives
		\begin{equation}\label{eq:lin_prob_kn0}
		\begin{aligned}
			D_{\bm{\eta}}F_{j}[0,\bm{\tau}]\left( \hat{\bm\eta}(\bm{k}) e^{i\bm{k}\cdot\bm{x}'}\right)&=\bigg(r^2\beta_{j}(\bm{k})\bigg[\Big(-\beta_{j}(\bm{k})[\rho_{j+1}\psi_{j+1}'(d_{j+1},\bm{k})+\rho_j\psi_{j}'(d_j,\bm{k})]\\
			&\qquad\qquad\qquad\qquad+\beta_{j}^\perp(\bm{k})[\rho_j\alpha_j-\rho_{j+1}\alpha_{j+1}]\Big)\hat{\eta}_{j}(\bm{k})\\
			&\qquad\qquad\qquad+\beta_{j+1}(\bm{k})\rho_{j+1}\psi_{j+1}'(d_j,\bm{k})\hat{\eta}_{j+1}(\bm{k})\\
			&\qquad\qquad\qquad+\beta_{j-1}(\bm{k})\rho_j\psi_{j}'(d_{j-1},\bm{k})\hat{\eta}_{j-1}(\bm{k})\bigg]\\
			&\qquad+(\rho_{j}-\rho_{j+1})g\hat{\eta}_{j}(\bm{k})+\sigma_{j}\vert\bm{k}\vert^2\hat{\eta}_{j}(\bm{k})\bigg)e^{i\bm{k}\cdot\bm{x}'},
		\end{aligned}
		\end{equation}
		where we have used the identities $ \phi_j'(d_{j-1},\bm{k})=-\psi_j'(d_j,\bm{k}) $ and $ \phi_j'(d_{j},\bm{k})=-\psi_j'(d_{j-1},\bm{k}) $ to replace all instances of $ \phi'_j $ with $ \psi'_j $.
		\begin{lemma}\label{lemma:Aproperties}
		The Fréchet derivative of $\bm{F}$ can be expressed as
		\[ 
			D_{\bm{\eta}}\bm{F}[0,\bm{\tau}](\bm{\eta})=\sum_{\bm{k}\in\Lambda'}D_{\bm{\eta}}\bm{F}[0,\bm{\tau}]\left(\hat{\bm{\eta}}(\bm{k})e^{i\bm{k}\cdot\bm{x}'}\right)=	\sum_{\bm{k}\in\Lambda'}A(\bm{\tau},\bm{k})\hat{\bm{\eta}}(\bm{k})e^{i\bm{k}\cdot\bm{x}'},
		\]
		where $ A(\bm{\tau},\bm{k}) $ is a $ n\times n $ matrix defined by \cref{eq:lin_prob_k0,eq:lin_prob_kn0}. In particular, $ A(\bm{\tau},\bm{k}) $ is a tridiagonal symmetric matrix.
	\end{lemma}
\begin{proof}
	Because $D_{\bm{\eta}}F_{j}[0,\bm{\tau}]( \hat{\bm\eta}(\bm{k}) e^{i\bm{k}\cdot\bm{x}'})$ only depends on $ \hat{\eta}_{j-1}(\bm{k}) $, $ \hat{\eta}_j(\bm{k}) $ and $ \hat{\eta}_{j+1}(\bm{k}) $ the matrix is tridiagonal. To show that the matrix is symmetric we consider
	\begin{align*}
	a_{j,j+1}&=\rho_{j+1}r^2\beta_{j}(\bm{k})\beta_{j+1}(\bm{k})\psi'_{j+1}(d_j,\bm{k}),\\
	a_{j+1,j}&=\rho_{j+1}r^2\beta_{j}(\bm{k})\beta_{j+1}(\bm{k})\psi'_{j+1}(d_j,\bm{k}),
	\end{align*}
	that is, $ a_{j,j+1}=a_{j+1,j} $. Hence the matrix is symmetric.
\end{proof}

\section{Existence result}\label{sec:existence}
We will begin with stating an assumption on the matrix $ A $ from \cref{lemma:Aproperties} that will immediately allow us to apply \Cref{thm:bif} in \Cref{sec:AppendixBif} to obtain an existence result.
\begin{assumption}\label{assumption:first}
	There exists $\bm{\tau}^*\in\mathfrak{Z}$ such that:
	\begin{itemize}
		\item[(i)] $\ker A(\bm{\tau}^*,\bm{k}_1)=\spann\{\bm{\hat{\eta}}(\bm{k}_1)\}$ and $\ker A(\bm{\tau}^*,\bm{k}_2)=\spann\{\bm{\hat{\eta}}(\bm{k}_2)\}$.
		\item[(ii)] The matrix with entries $ \nu_{i,l} $ given by
		\begin{align*}
		\nu_{i,1}=\bm{\hat{\eta}}(\bm{k}_i)\cdot \partial_r A(\bm{\tau}^*,\bm{k}_i)\bm{\hat{\eta}}(\bm{k}_i)\\
		\nu_{i,2}=\bm{\hat{\eta}}(\bm{k}_i)\cdot \partial_{\theta} A(\bm{\tau}^*,\bm{k}_i)\bm{\hat{\eta}}(\bm{k}_i)
		\end{align*}
		is invertible.
		\item[(iii)]
		\[
		\det A(\bm{\tau}^*,\bm{k})\neq 0 \qquad \qquad\text{for all }\bm{k}\in\Lambda':\bm{k}\neq \pm \bm{k}_i, \qquad i=1,2,
		\]
	\end{itemize}
\end{assumption}
\begin{remark}
	Since $A(\bm{\tau},\bm{k})=A(\bm{\tau},-\bm{k})$ we immediately get $\ker A(\bm{\tau}^*,-\bm{k}_1)=\spann\{\bm{\hat{\eta}}(\bm{k}_1)\}$ and $\ker A(\bm{\tau}^*,-\bm{k}_2)=\spann\{\bm{\hat{\eta}}(\bm{k}_2)\}$ from part \textit{(i)} of this assumption. Together with part \textit{(iii)} this implies
	\[
	\ker D_{\bm{\eta}}\bm{F}[0,\bm{\tau}^*]=\spann\{\bm{\hat{\eta}}(\bm{k}_1)\cos(\bm{k}_1\cdot \bm{x}'),\bm{\hat{\eta}}(\bm{k}_2)\cos(\bm{k}_2\cdot \bm{x}')\},
	\]
	because $D_{\bm{\eta}}\bm{F}[0,\bm{\tau}]$ is an operator the space $\mathfrak{Y}$ consisting of even functions
\end{remark}
\Cref{assumption:first} is specifically designed to make $ \bm{F} $ from  \cref{eq:SingleSurf} satisfy the assumptions in \Cref{thm:bif}. So proving an existence result at this point under the assumption is rather straightforward. The difficulty lies in proving the validity of the assumption itself, which is a problem studied in \Cref{sec:assumption}. However, the analysis of the assumption in this paper is not completely exhaustive, which means that the existence theorem below remains valid in a larger subset of the parameter space than shown in \Cref{sec:assumption}. With this assumption in hand we are ready to prove the first and more abstract version of the main result.

\begin{theorem}\label{thm:exist} If \cref{assumption:first} and \cref{eq:no-resonance} holds then there exits an $ \epsilon>0 $ such that for every $\bm{t}=(t_1,t_2)\in B_{\epsilon}(0)$ there exist parameters $\bm{\tau}(\bm{t})\in \mathfrak{Z}$ and corresponding solutions $\bm{v}(\bm{t})\in \mathfrak{X}^0$ and $\bm{\eta}(\bm{t})\in \mathfrak{Y}$ to \crefrange{eq:curl_flat}{eq:dyn_flat} such that the map
\[
	\bm{t}\mapsto (\bm{\tau}(\bm{t}),\bm{v}(\bm{t}),\bm{\eta}(\bm{t}))
\]
is real analytic and
\[
	\bm{\tau}=\bm{\tau}^*+\mathcal{O}(\vert \bm{t}\vert),\qquad \bm{\eta}=t_1\hat{\bm{\eta}}(\bm{k}_1)\cos(\bm{k}_1\cdot \bm{x}')+t_2\hat{\bm{\eta}}(\bm{k}_2)\cos(\bm{k}_2\cdot\bm{x}')+\mathcal{O}(\vert \bm{t}\vert^2).
\]
In particular, this means that there exist nontrivial solutions to \crefrange{eq:curl}{eq:int_cond}.
\end{theorem}
\begin{proof} We prove the theorem by checking conditions \textit{(i)}--\textit{(iv)} of \Cref{thm:bif} for 
	\[
	\bm{F}:\mathfrak{Y}\times \mathfrak{Z}\mapsto \mathfrak{W}
	\] from \cref{eq:SingleSurf}. Where $ \mathcal{X}_1, \mathcal{X}_2 $ and $ \mathcal{Y}_1,\mathcal{Y}_2 $ are the subspaces of $\mathfrak{Y} $ and $ \mathfrak{W} $ spanned by $ x_1=y_1=\hat{\bm{\eta}}(\bm{k}_1)\cos(\bm{k}_1\cdot \bm{x}') $ and $ x_2=y_2=\hat{\bm{\eta}}(\bm{k}_2)\cos(\bm{k}_2\cdot \bm{x}') $ respectively. It follows from the fact that $ A $ is symmetric that the kernel and cokernel of $ D_{\bm{\eta}}\bm{F}[0,\bm{\tau}^*] $ are spanned by the same functions.
	
	Condition \textit{(i)} of \Cref{thm:bif} is trivially satisfied.
That the derivative $ D_{\bm{\eta}}\bm{F}[0,\bm{\tau}^*]$ is a well defined linear bounded operator mapping $\mathfrak{Y}$ to $\mathfrak{W} $ follows from the fact that $ A(\bm{\tau},\bm{k})\hat{\bm{\eta}}(\bm{k})\sim \vert \bm{k}\vert^2\hat{\bm{\eta}}(\bm{k})$ for large $ \vert\bm{k}\vert $. Moreover, by part \textit{(i)} and \textit{(iii)} of \Cref{assumption:first} the kernel is two dimensional and given by $\hat{\mathcal{X}}=\mathcal{X}_1\oplus \mathcal{X}_2 $ as defined above and the cokernel is also two dimensional and given by $\hat{\mathcal{Y}}=\mathcal{Y}_1\oplus \mathcal{Y}_2 $. Hence it is a Fredholm operator and condition \textit{(ii)} of \Cref{thm:bif} is satisfied. Condition \textit{(iii)} of \Cref{thm:bif} is exactly the same as part \textit{(ii)} of \Cref{assumption:first}. Finally, condition \textit{(iv)} of \Cref{thm:bif} is satisfied if we pick $\tilde{\mathcal{X}}_i$ and $\tilde{\mathcal{Y}}_i$ such that $\hat{\mathcal{X}}_{i}\oplus \tilde{\mathcal{X}}_i $ and $ \hat{\mathcal{Y}}_{i}\oplus \tilde{\mathcal{Y}}_i $ are the subspaces of functions that are constant in the $ \bm{\lambda}_l $ direction for $ i\neq l $. 
\end{proof}
\section{On \Cref{assumption:first}}\label{sec:assumption}
This section is dedicated to show that \Cref{assumption:first} is satisfied in large parts of the parameter space of the problem. To this end we note that, due to fact that $ A(\bm{\tau},\bm{k}) $ is tridiagonal, it is sufficient to find $ \bm{\tau}^* $ such that $ \det A(\bm{\tau}^*,\bm{k}_1)=\det A(\bm{\tau}^*,\bm{k}_2)=0 $ and check that the sub- and superdiagonal consist of nonzero elements for part $ (i) $ of $ \Cref{assumption:first} $. Then to check parts $ (ii) $ and $ (iii) $ it is sufficient to check that a finite number of quantities are nonzero. This is due to the fact that for large $ \vert \bm{k}\vert $ the terms $ \sigma_j\vert \bm{k}\vert^2 $ are dominant so $ \det A(\bm{\tau}^*,\bm{k})\neq 0 $ for all sufficiently large $ \vert \bm{k}\vert $.

The results here are by no means exhaustive, since the parameter space is very large and a complete characterization of the subset where \Cref{assumption:first} is satisfied may not even be possible to do in a simple way. We focus on giving results with relatively simple assumptions in this section, because it is not difficult to check numerically if \Cref{assumption:first} is satisfied for a given set of parameter values. Thus a characterization of the the set where \Cref{assumption:first} is satisfied in terms of equally complicated algebraic expressions would be of little use. 

For this section we also make the following definition
\begin{definition}
	We call a lattice $ \Lambda\subset \mathbb{R}^2$ a \emph{symmetric lattice} if its generators $ \bm{\lambda}_1 $ and $ \bm{\lambda}_2 $ satisfy
	\[ 
	\vert \bm{\lambda}_1\vert =\vert \bm{\lambda}_2\vert.
	\]
	Moreover, we call it \emph{non-degenerate} if
	\[ 
	\vert \bm{\lambda}\vert \neq \vert \bm{\lambda}_1\vert\qquad\text{for all } \bm{\lambda}\in \Lambda\setminus\{\pm\bm{\lambda}_1,\pm\bm{\lambda}_2\}.
	\]
\end{definition}
The conditions on the lattice are exclusively required for part $ (iii) $ of \cref{assumption:first}. Even under this assumption we may still have to redefine $ \bm{\sigma} $ in some arbitrarily small neighborhood. However, it should be noted that it is plausible part $ (iii) $ of \cref{assumption:first} is true for a general lattice, $ \bm{\tau}^* $ such that parts $ (i) $ and $ (ii) $ of \Cref{assumption:first} is satisfied, and almost all parameter values. As stated above, part $ (iii) $ of \Cref{assumption:first} is equivalent to a finite subset of $ \Lambda' $ does not solve $ \det A(\bm{\tau}^*,\bm{k})=0 $. For any given set of parameters it is unlikely that any of these $ \bm{k} $ solve $ \det A(\bm{\tau}^*,\bm{k})=0 $. However, we can expect this to happen for some set of particular parameter values. Although in general we expect this set to be very small.
\subsection{The case $n=1$}
The case $ n=1 $ and $ \rho_2=0 $, was considered in \cite{Lokharu_2020}. The case when $ \rho_2>0 $ can be treated similarly. To make this paper more self-contained we give an adapted result below; prescribing sufficient conditions for \cref{assumption:first} to be satisfied in the latter case. For the former case and a more detailed discussion we refer to \cite{Lokharu_2020}. When $ n=1 $ the matrix $ A(\bm{k},\bm{\tau}) $ is simply a scalar and if, in addition, $ \rho_2>0 $ it is given by
\begin{align*}
	A(\bm{k},\bm{\tau})&=r^2\beta_{1}(\bm{k})\Big(-\beta_{1}(\bm{k})[\rho_2\psi_{2}'(d_{2},\bm{k})+\rho_1\psi_{1}'(d_1,\bm{k})]
	+\beta_{1}^\perp(\bm{k})[\rho_1\alpha_1-\rho_2\alpha_{2}]\Big)\\
	&\qquad+(\rho_{1}-\rho_{2})g+\sigma_{1}\vert\bm{k}\vert^2.
\end{align*}
For the following result we shall assume that $ \gamma_1=0 $ and $ 0<\gamma_2<\pi $. This can be done without loss of generality because we can always rotate the coordinate system and change $ \bm{k}_2 $ to $ -\bm{k}_2 $ if necessary.
This equation can be analyzed as a hyperbola in the variables $ x=r\beta_1(\bm{k}) $ and $ y=r\beta_1^\perp(\bm{k}) $, and we obtain the result below.
\begin{proposition}[\cite{Lokharu_2020}]
\label{thm:existsn1}
If
\begin{itemize}
	\item[(i)] $ \rho_2\psi_{2}'(d_{2},\bm{k}_1)+\rho_1\psi_{1}'(d_1,\bm{k}_1)\geq \rho_2\psi_{2}'(d_{2},\bm{k}_2)+\rho_1\psi_{1}'(d_1,\bm{k}_2)>0 $, and $ \rho_1\alpha_1-\rho_2\alpha_{2}>0 $;
	\item[(ii)] $ \arctan\left(\frac{\rho_2\psi_{2}'(d_{2},\bm{k}_1)+\rho_1\psi_{1}'(d_1,\bm{k}_1)}{\rho_1\alpha_1-\rho_2\alpha_{2}}\right)-\arctan\left(\frac{\rho_2\psi_{2}'(d_{2},\bm{k}_2)+\rho_1\psi_{1}'(d_1,\bm{k}_2)}{\rho_1\alpha_1-\rho_2\alpha_{2}}\right)<\gamma_2$;
\end{itemize}
then there exist $ \bm{\tau}^* $ such that part $ (i) $ and $ (ii) $ of $ \cref{assumption:first} $ are satisfied;
if, additionally,
\begin{itemize}
	\item[(iii)] $ \Lambda' $ is a symmetric lattice;
\end{itemize}
then part $ (iii) $ of \cref{assumption:first} is also satisfied for $ \bm{\tau}^* $ and almost every $ \sigma_1 $.
\end{proposition}
\begin{remark}
	The first condition is more or less a matter of proper labeling, but the second condition requires us to be in this setting. We can also replace the conditions $ (i) $ and $ (ii) $, by $ \rho_1\alpha_1=\rho_2\alpha_2 $. See \cite{Lokharu_2020} for further details. 
\end{remark}
This proposition can be combined with \Cref{thm:exist} to obtain solutions under the assumptions given in the proposition.
\subsection{Irrotational and small vorticity}
In the previous section we only considered the case $ n=1 $. Here we also allow any $ n\geq 1 $. Unfortunately, lifting the restriction $ n=1 $ leads to a substantial increase of complexity and we have to impose some other restrictions; specifically small vorticity. The result of this section is the second main theorem of this paper and gives an existence result without supposing \Cref{assumption:first} holds. Instead more direct assumptions on the parameters are given, which in turn imply that \Cref{assumption:first} holds. This gives us the second and more concrete version of the main result. However, before we state this theorem we make the following observation to introduce some necessary notation. If $\bm{\alpha}=0$, then $\beta_j(\bm{k})=\cos(\theta-\gamma)$, $1\leq j\leq m$, so we can write
\begin{equation}\label{eq:AnoVort}  
A(\bm{\tau},\bm{k})=\Sigma(k)+r^2\cos^2(\theta-\gamma)B(k),
\end{equation}
where $ \Sigma(k) $ is a diagonal matrix with elements given by
\[ 
(\Sigma(k))_{j,j}=\sigma_jk^2+(\rho_j-\rho_{j+1})g>0.
\]
This allows us to state the main result of this section.
\begin{theorem}\label{thm:existssmallvort} Assume that $ \Lambda' $ is a non-degenerate symmetric lattice, the non-resonance condition in \cref{eq:no-resonance} holds, and that $\vert \bm{\alpha}\vert\ll 1 $. 
	Then, after possibly redefining $ \bm{\sigma} $ in an arbitrarily small neighborhood, there exist $ n^2 $ parameter values $ \bm{\tau}_{\iota,\kappa}^* \in \mathfrak{Z}$, $ 1\leq \iota,\kappa\leq n $, and an $ \epsilon $ such that for every $\bm{t}=(t_1,t_2)\in B_{\epsilon}(0)$ there exist parameters $ \bm{\tau}_{\iota.\kappa}(\bm{t})\in\mathfrak{Z} $ and corresponding solutions $\bm{v}_{\iota,\kappa}(\bm{t})\in \mathfrak{X}$, and $\bm{\eta}_{\iota,\kappa}(\bm{t})\in \mathfrak{Y}$ to \crefrange{eq:curl_flat}{eq:dyn_flat} such that the map
	\[
	\bm{t}\mapsto (\bm{\tau}(\bm{t}),\bm{v}_{\iota,\kappa}(\bm{t}),\bm{\eta}(\bm{t}))
	\]
	is real analytic and
	\[
	\bm{\tau}_{\iota,\kappa}=\bm{\tau}^*_{\iota,\kappa}+\mathcal{O}(\vert \bm{t}\vert),\qquad \bm{\eta}_{\iota,\kappa}=t_1\hat{\bm{\eta}}_{\iota}(\bm{k}_1)\cos(\bm{k}_1\cdot \bm{x}')+t_2\hat{\bm{\eta}}_{\kappa}(\bm{k}_2)\cos(\bm{k}_2\cdot\bm{x}')+\mathcal{O}(\vert \bm{t}\vert^2).
	\]
	Moreover, the vectors $ \hat{\bm{\eta}}_\iota(\bm{k}_1) $ and $ \hat{\bm{\eta}}_\kappa(\bm{k}_2) $ that span $ \ker{A}(\bm{k}_1,\bm{\tau}_{\iota,\kappa}) $ and $ \ker{A}(\bm{k}_1,\bm{\tau}_{\iota,\kappa}) $ respectively, are also given by
	\[ 
	\hat{\bm{\eta}}_\iota(\bm{k}_1)=\Sigma^{-1/2}(k_1)\bm{\xi}_\iota(k_1)+\mathcal{O}(\vert\bm{\alpha}\vert),\qquad \hat{\bm{\eta}}_\kappa(\bm{k}_2)=\Sigma^{-1/2}(k_2)\bm{\xi}_\kappa(k_2)+\mathcal{O}(\vert\bm{\alpha}\vert),
	\]
	where $ \bm{\xi}_{\iota}(k_1) $ and $ \bm{\xi}_{\kappa}(k_2) $ are eigenvectors of $ \Sigma^{-1/2}(k_1)B(k_1)\Sigma^{-1/2}(k_1) $ and\\ $ \Sigma^{-1/2}(k_2)B(k_2)\Sigma^{-1/2}(k_2) $ with corresponding eigenvalues $ \mu_{\iota}(\bm{k}_1) $, $ \mu_{\kappa}(\bm{k}_2) $ satisfying
	\begin{align*}
	\mu_{\iota}(\bm{k}_1)=-\frac{1}{r_{\iota,\kappa}^2\cos^2(\theta_{\iota,\kappa}-\gamma_1)},\qquad
	\mu_{\kappa}(\bm{k}_2)=-\frac{1}{r_{\iota,\kappa}^2\cos^2(\theta_{\iota,\kappa}-\gamma_2)}.
	\end{align*}
\end{theorem}
\begin{remark}
	This theorem means that there exist $ n^2 $ different nontrivial solutions to \crefrange{eq:curl}{eq:int_cond}. In fact we expect there to exist more than $n^2$ values of $\bm{\tau}$ that yield solutions. However, these additional $\bm{\tau}$ values give rise to solutions that are qualitatively similar to the ones in the Theorem. For $ n=2 $ we illustrate and example of the two modes in \Cref{fig:n2modes} and the four different waves they can be combined to in \Cref{fig:n2waves}. It should also be noted that for $ \bm{\alpha}=0 $, the solution for $ (\iota,\kappa) $ is a reflection of the solution for $ (\kappa,\iota) $. Consider for example a square lattice with $ \bm{k}_1=(k,0) $ and $ \bm{k}_2=(0,2) $. Let
	\begin{align*} 
	S_j((u^{(j)}_1,u^{(j)}_2,u^{(j)}_3)(x,y,z))
	&=((u^{(j)}_2,u^{(j)}_1,u^{(j)}_3)(y,x,z)),\\
	\tilde{S}_j\eta_j(x,y)&=\eta_j(y,x),\\
	\intertext{and}
	\bm{S}(\bm{u},\bm{\eta})&=(S_1\bm{u}^{(1)},\ldots,S_m\bm{u}^{(m)},\tilde{S}_1\eta_1(x,y),\ldots,\tilde{S}_n\eta_n(x,y)),
	\end{align*}
	then $ \bm{S}(\bm{u}_{\iota,\kappa},\bm{\eta}_{\iota,\kappa})[t_1,t_2]=(\bm{u}_{\kappa,\iota},\bm{\eta}_{\kappa,\iota})[t_2,t_1] $. For $ \bm{\alpha}\neq 0 $ we still get a solution to the Euler equations under this reflection, but the reflected velocity fields, $ S_j\bm{u}^{(j)} $, are Beltrami fields with constants $-\alpha_j $; that is, if $ \nabla \times \bm{u}^{(j)}=\alpha_j \bm{u}^{(j)} $ then $ \nabla \times S_j\bm{u}^{(j)}=-\alpha_j S_j\bm{u}^{(j)} $.
\end{remark}
\begin{figure}[h!]
	\begin{centering}
		\begin{subfigure}{0.45\textwidth}
			\includegraphics[scale=0.55]{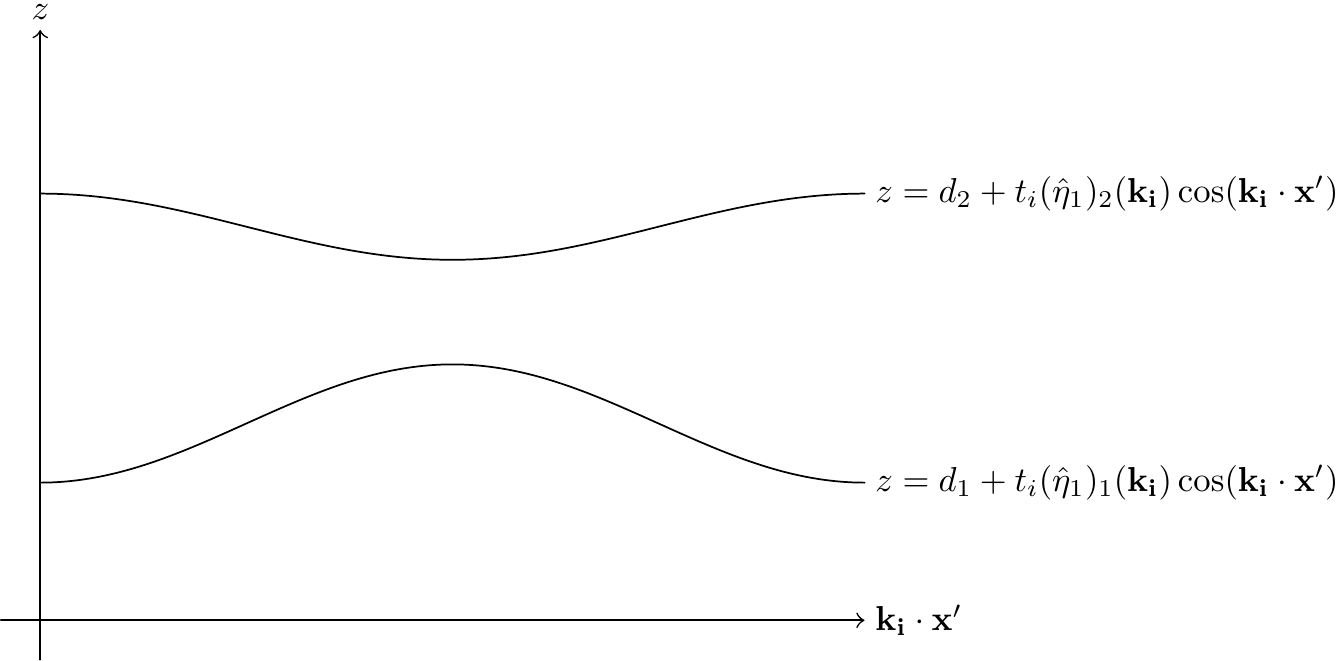}
			\caption{The mode given by $ \hat{\bm{\eta}}_1(\bm{k}_i)  $. The interfacial waves are out of phase.}
		\end{subfigure}
		\hfill
		\begin{subfigure}{0.45\textwidth}
			\includegraphics[scale=0.55]{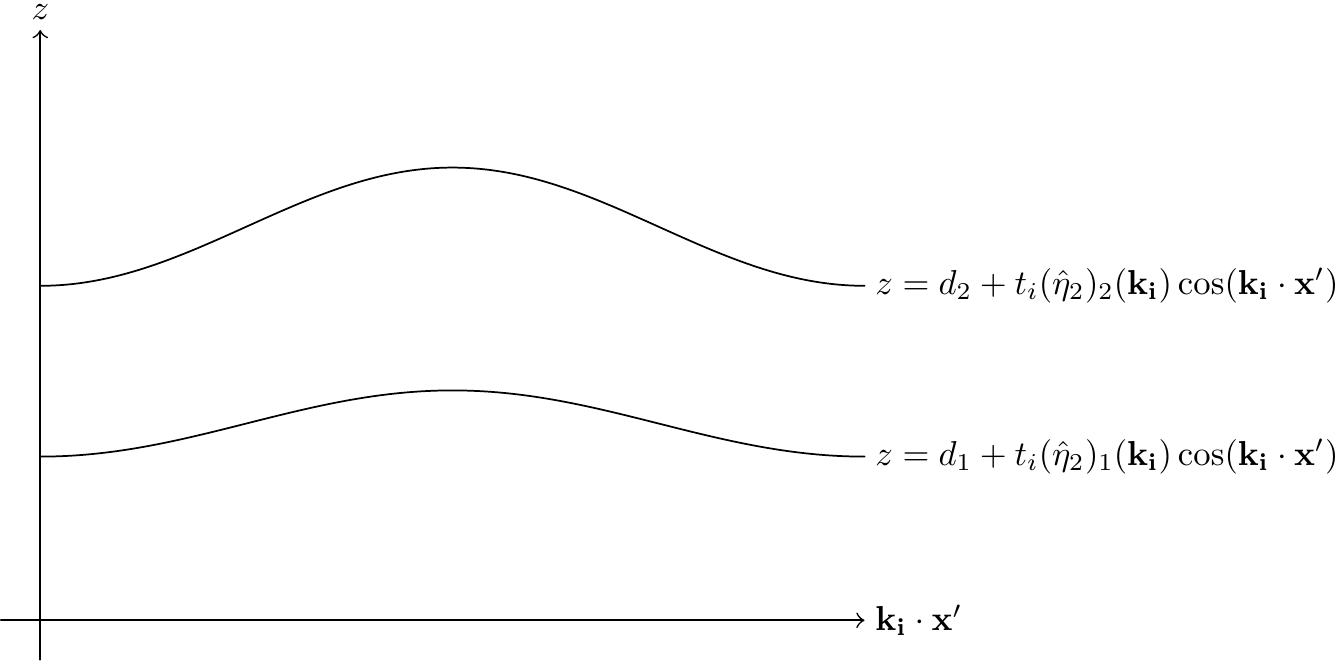}
			\caption{The mode given by $ \hat{\bm{\eta}}_2(\bm{k}_i)  $. The interfacial waves are in phase.}
		\end{subfigure}
	\end{centering}
	\caption{This figure shows an example of the two different modes that can be obtained for $n=2$. That is, a two-dimensional cross section of the domain highlighting $ \hat{\bm{\eta}}_1(\bm{k}_i) $ and $ \hat{\bm{\eta}}_2(\bm{k}_i) $ respectively, with the interfaces expanded to first order in $ \bm{t} $.}
	\label{fig:n2modes}
\end{figure}
\begin{figure}[h!]
	\begin{centering}
		\begin{subfigure}{0.45\textwidth}
			\includegraphics[scale=0.6]{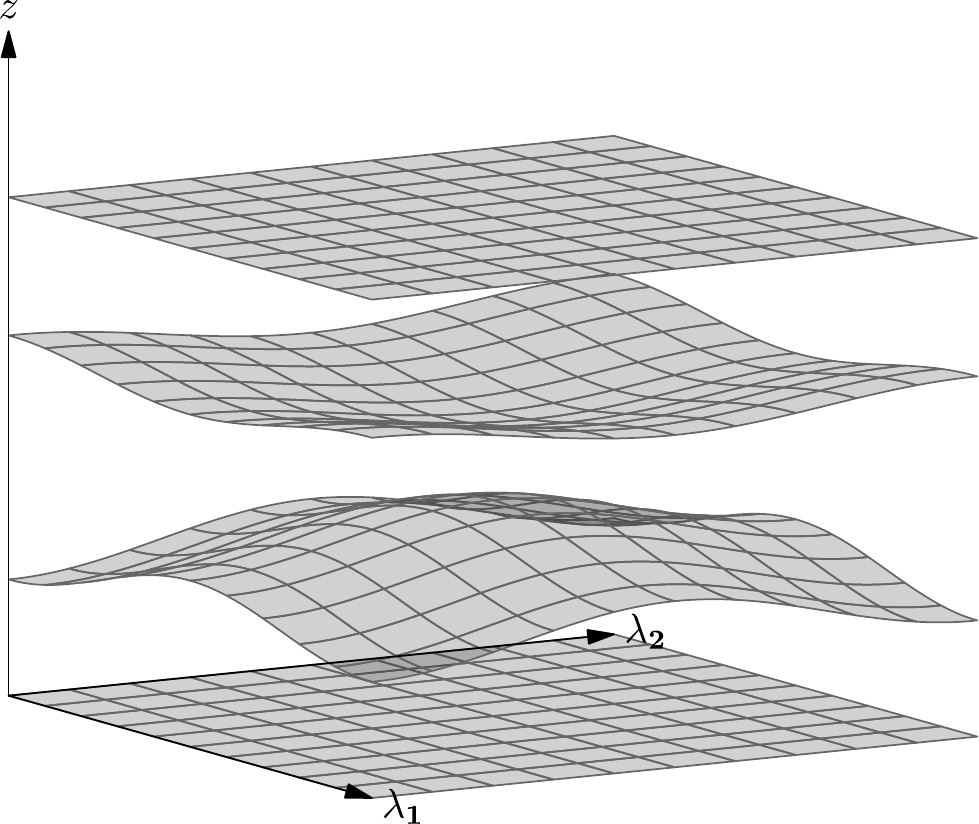}
			\caption{The $ \bm{\eta}_{1,1} $-waves. The interfacial waves are out of phase in both directions.}
		\end{subfigure}
		\hfill
		\begin{subfigure}{0.45\textwidth}
			\includegraphics[scale=0.6]{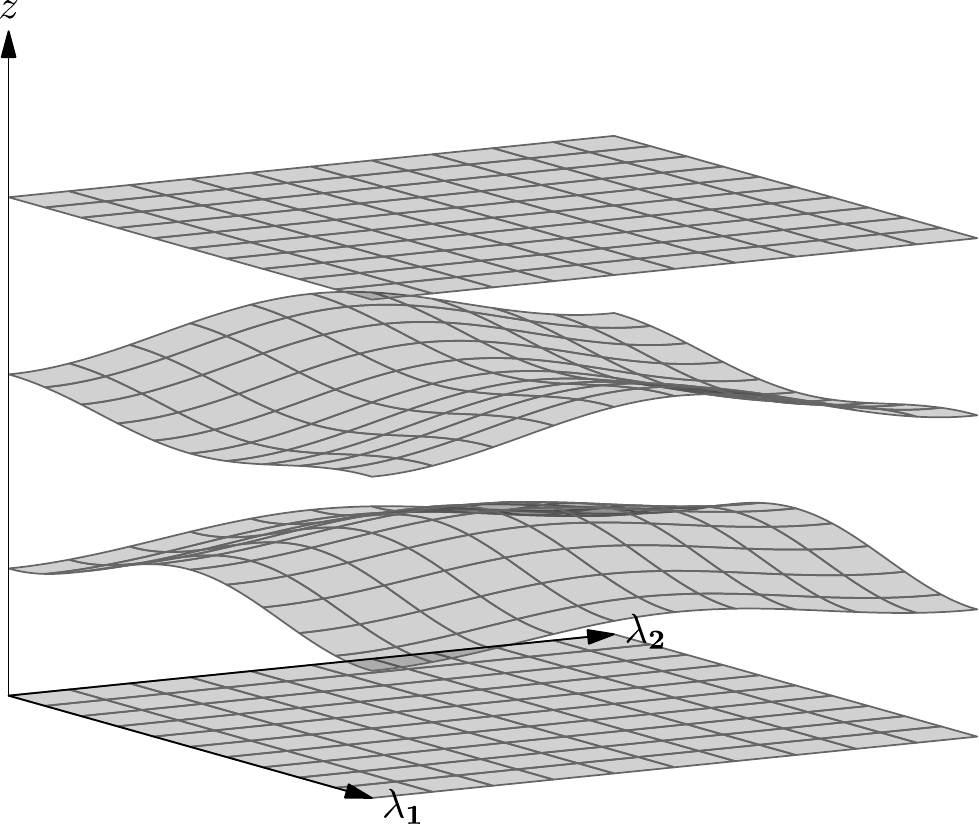}
			\caption{The $ \bm{\eta}_{1,2} $-waves. The interfacial waves are out of phase in the $ \bm{\lambda}_1 $-direction and in phase in the $ \bm{\lambda}_2$-direction.}
			\label{subfig:12-wave}
		\end{subfigure}
		\begin{subfigure}{0.45\textwidth}
			\includegraphics[scale=0.6]{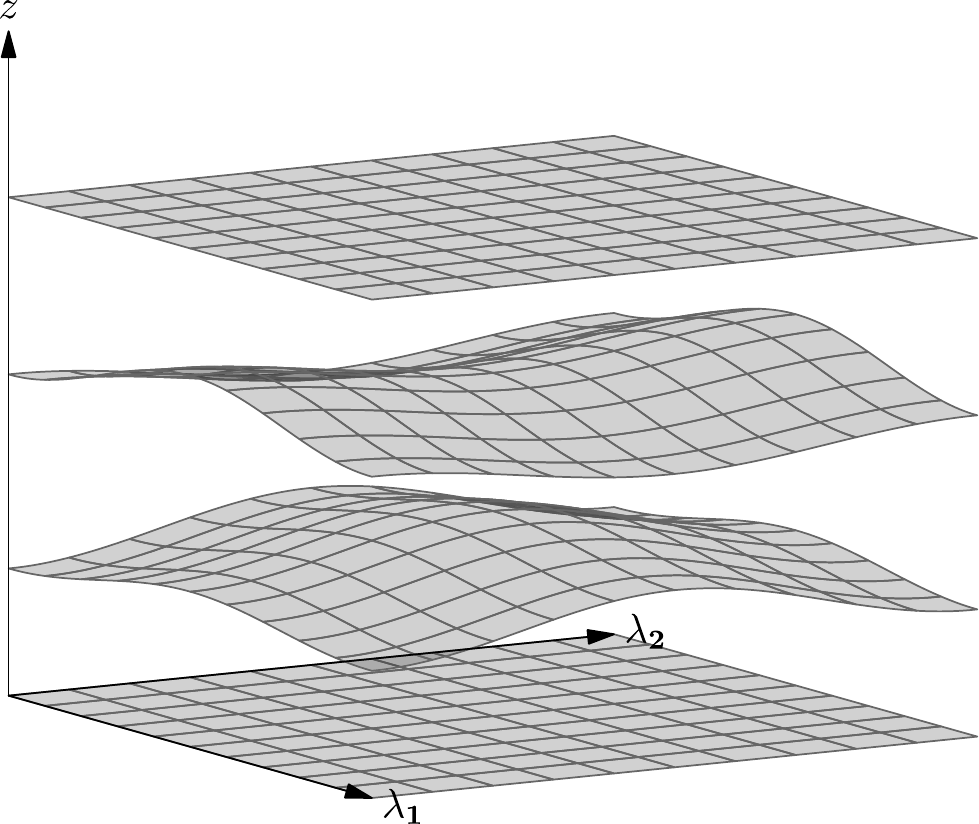}
			\caption{The $ \bm{\eta}_{2,1} $-waves. The interfacial waves are in phase in the $ \bm{\lambda}_1 $-direction and out of phase in the $ \bm{\lambda}_2$-direction.}
			\label{subfig:21-wave}
		\end{subfigure}
		\hfill
		\begin{subfigure}{0.45\textwidth}
			\includegraphics[scale=0.6]{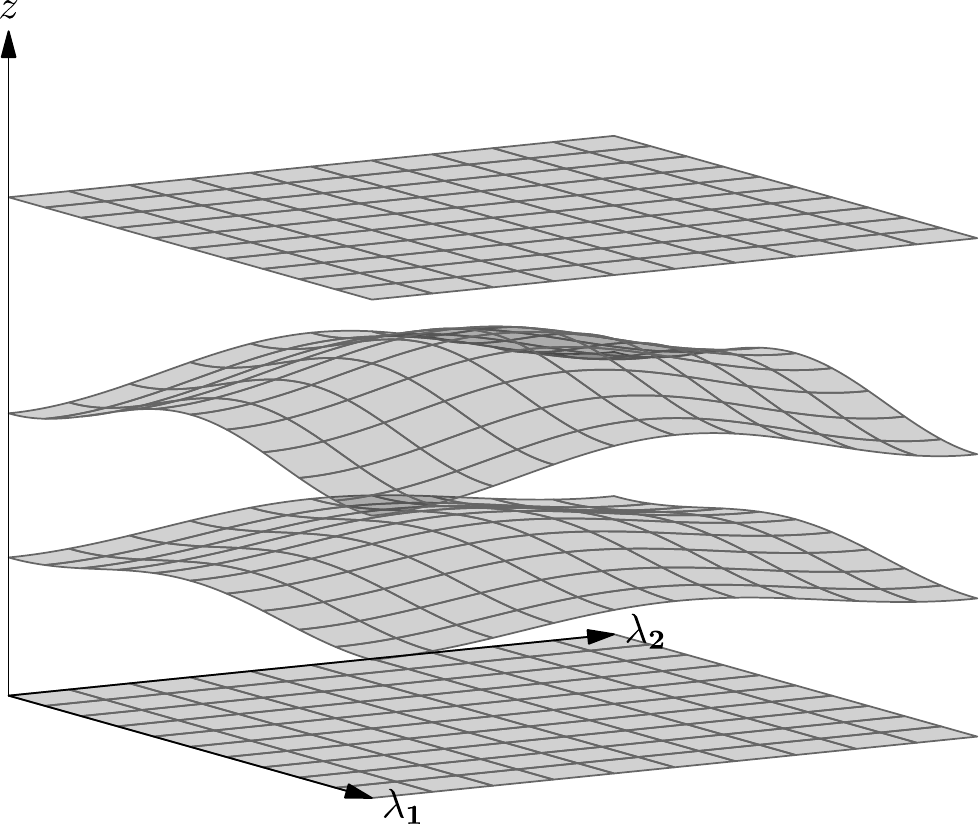}
			\caption{The $ \bm{\eta}_{2,2} $-waves. The interfacial waves are in phase in both directions.}
		\end{subfigure}
	\end{centering}
	\caption{This figure shows an example of the four different possible interfacial waves $ \bm{\eta}_{\iota,\kappa} $, for $ \iota=1,2 $ and $ \kappa=1,2 $, when $n=2$. The solutions are expanded to first order in $ \bm{t} $ and can be obtained by combining the modes in \Cref{fig:n2modes}. Note that while the $\bm{\eta}_{2,1}$-waves in \Cref{subfig:21-wave} can be obtained by the $\bm{\eta}_{1,2}$-waves in \Cref{subfig:12-wave} through a reflection, the reflected velocity field is different unless $ \bm{\alpha}=0 $.}
	\label{fig:n2waves}
\end{figure}
The rest of this section is dedicated to the proof of this theorem.
We begin working under the assumption $ \bm{\alpha}=0 $ and extend this to $ \vert\bm{\alpha}\vert\ll1 $ through an implicit function theorem argument. In the proofs we will also make use of submatrices and thus make the following definition.
	\begin{definition}
		For an $ n\times n $ matrix $ M $ we define $ M_{k} $ as the $ k\times k $ matrix obtained by removing the last $ n-k $ rows and the last $ n-k $ columns. We call $ M_k $ a \emph{leading principal submatrix} of $ M $ and $ \det M_{k} $ a \emph{leading principal minor} of $ M $.
	\end{definition}
Now we can prove the following properties of the matrix $ B(k) $.
	\begin{lemma}\label{lemma:ApropertiesIrrotational}
		If $ \bm{\alpha}=0 $ and $\bm{k}\neq 0$, then the matrix $ B(k) $ in \cref{eq:AnoVort} has negative elements on the diagonal, positive elements on the sub- and superdiagonal, and all eigenvalues of $ B(k) $ are negative.
	\end{lemma}
\begin{proof}
	If $ \bm{\alpha}=0 $ then $ \beta_j=\cos(\theta-\gamma) $, which means we can write
	\[ 
		A(\bm{\tau},\bm{k})=\Sigma(k)+r^2\cos^2(\theta-\gamma)B(k).
	\]
	$ \Sigma(k) $ is a diagonal matrix with diagonal elements given by
	\[ 
		(\Sigma(k))_{j,j}=\sigma_jk^2+(\rho_j-\rho_{j+1})g>0,
	\]
	and $ B(k) $ is tridiagonal with nonzero elements given by
	\begin{align*} 
		(B(k))_{j,j-1}&=\rho_j\psi'_j(d_{j-1},k)>0,\\
		(B(k))_{j,j}&= -\rho_{j+1}\psi'_{j+1}(d_{j+1},k)-\rho_j\psi'_j(d_j,k)<0,\\
		(B(k))_{j,j+1}&=\rho_{j+1}\psi'_{j+1}(d_j,k)>0.
	\end{align*}
	Moreover, $ B(k) $ can be further decomposed into
	$ B(k)=\hat{B}(k)+\check{B}(k) $, where
	\begin{align*}
		(\hat{B}(k))_{j,j-1}&=
		\begin{cases}
			0 &\text{$ j $ odd},\\
			\rho_j\psi'_j(d_{j-1},k) &\text{$ j $ even},
		\end{cases}\\
		(\hat{B}(k))_{j,j}&=
		\begin{cases}
			-\rho_{j+1}\psi'_{j+1}(d_{j+1},k)&\text{$ j $ odd},\\
			-\rho_j\psi'_j(d_j,k)&\text{$ j $ even},
		\end{cases}\\
		(\hat{B}(k))_{j,j+1}&=
		\begin{cases}
			\rho_{j+1}\psi'_{j+1}(d_j,k) &\text{$ j $ odd},\\
			0 &\text{$ j $ even},
		\end{cases}
	\end{align*}
	and
	\begin{align*}
		(\check{B}(k))_{j,j-1}&=
		\begin{cases}
			\rho_j\psi'_j(d_{j-1},k) &\text{$ j $ odd},\\
			0 &\text{$ j $ even},
		\end{cases}\\
		(\check{B}(k))_{j,j}&=
		\begin{cases}
			-\rho_j\psi'_j(d_j,k)&\text{$ j $ odd},\\
			-\rho_{j+1}\psi'_{j+1}(d_{j+1},k)&\text{$ j $ even},
		\end{cases}\\
		(\check{B}(k))_{j,j+1}&=
		\begin{cases}
			0 &\text{$ j $ odd},\\
			\rho_{j+1}\psi'_{j+1}(d_j,k) &\text{$ j $ even}.
		\end{cases}
	\end{align*}
	 Computing the leading principal minors of $ -\hat{B}(k) $ and $ -\check{B}(k) $ is straightforward. We obtain
	\[ 
		\det(-\hat{B}_{l}(k))=
		\begin{cases}
			k^{l}\prod_{i=1}^{\frac{l}{2}}\rho_{2i}^2 & \text{$l$ even},\\
			k^{l}\rho_{l+1}\psi'_{l+1}(d_{l+1},k)\prod_{i=1}^{\frac{l-1}{2}}\rho_{2i}^2 & \text{$l$ odd},
		\end{cases}
	\]
	and
	\[  
		\det(-\check{B}_{l}(k))=
		\begin{cases}
			k^{l}\rho_1 \psi'_1(d_1,k)\rho_{l+1}\psi'_{l+1}(d_{l+1},k)\prod_{i=1}^{\frac{l-2}{2}}\rho_{2i+1}^2 & \text{$l$ even},\\
			k^{l}\rho_1 \psi'_1(d_1,k)\prod_{i=1}^{\frac{l-1}{2}}\rho_{2i+1}^2 & \text{$l$ odd}.
		\end{cases}
	\]
	Both being positive for all $ 1\leq l\leq n $ if $ \rho_{n+1}>0 $, which means $ -\hat{B}(k) $ and $ -\check{B}(k) $ are positive definite, whence $ \hat{B}(k) $ and $ \check{B}(k) $ are negative definite and so is their sum $ B(k) $. If $ \rho_{n+1}=0 $, then either $ \hat{B}(k) $ or $ \check{B}(k) $ have only zeros in the last row and column. However, $ \hat{B}_{n-1}(k) $ and $ \check{B}_{n-1}(k) $ are still negative definite. Thus one of $ \hat{B}(k) $ and $ \check{B}(k) $ is negative definite, while the other is negative semi-definite. The conclusion is that their sum is negative definite still holds in this case as well.
\end{proof}
With these properties we can prove that \Cref{assumption:first} $ (i) $ and $ (ii) $ are satisfied. In fact there are several $ \bm{\tau}^* $ for $ n>1 $.
\begin{proposition}\label{prop:As(I&II)Irrotational} If $ \bm{\alpha}=0 $, then there exist $ n^2 $ parameter values $ \bm{\tau}^*_{\iota,\kappa}\in\mathfrak{Z} $, $ 1\leq \iota,\kappa\leq n $, such that \Cref{assumption:first} (i) and (ii) are satisfied with distinct $(\hat{\bm{\eta}}_\iota(\bm{k}_1),\hat{\bm{\eta}}_\kappa(\bm{k}_2))= (\ker(A(\bm{\tau}^*_{\iota,\kappa},\bm{k_1})),\ker(A(\bm{\tau}^*_{\iota,\kappa},\bm{k_2})))$.
\end{proposition}
\begin{remark}
	It should be noted that the $ \bm{\tau}^*_{\iota,\kappa} $ for fixed $ \iota $ and $ \kappa $ is not unique in general. The proposition only ensures us that there exists at least one $ \bm{\tau}^*_{\iota,\kappa} $ for every possible choice of $ \iota $ and $ \kappa $.
\end{remark}
\begin{proof} By \Cref{lemma:ApropertiesIrrotational,lemma:Aproperties} the matrices $ \Sigma^{-1/2}(k_l)B(k_l)\Sigma^{-1/2}(k_l) $ are symmetric, tridiagonal, negative definite with nonzero elements on the sub- and superdiagonals. Thus there exist eigenvalues $\mu_\iota(k_1)<0$ and $ \mu_\kappa(k_2)<0$, and corresponding orthogonal eigenvectors $ \bm{\xi}_\iota(k_1)\neq 0 $ and $ \bm{\xi}_\kappa(k_2)\neq 0 $, $ 1\leq \iota,\kappa\leq n $. Moreover, $ \mu_{\iota}(k_l)\neq \mu_{\kappa}(k_l) $ if $ \iota\neq \kappa $ because
	\begin{equation}\label{eq:Aker1dim}
		\dim(\ker{(\Sigma^{-1/2}(k_l)B(k_l)\Sigma^{-1/2}(k_l)-\mu_{\iota}(k_l)I)})=1
	\end{equation}
	for all $ \iota $. This follows from the rank-nullity theorem and the fact that
	\[ 
		\dim(\im{(\Sigma^{-1/2}(k_l)B(k_l)\Sigma^{-1/2}(k_l)-\mu_{\iota}(k_l)I)})=n-1,
	\]
	because the elements of $ \Sigma^{-1/2}(k_l)B(k_l)\Sigma^{-1/2}(k_l)-\mu_{\iota}(k_l)I $ on the sub- and superdiagonals are nonzero.
	We obtain $ \bm{\tau}_{\iota,\kappa}^* $ by solving
	\begin{align*} 
		r^2\cos^2(\theta-\gamma_1)=-\frac{1}{\mu_{\iota}(k_1)}, \qquad
		r^2\cos^2(\theta-\gamma_2)=-\frac{1}{\mu_{\kappa}(k_2)}.
	\end{align*}
	It follows that $ \theta^*_{\iota,\kappa} $ is a solution to
	\[ 
		\frac{\cos^2(\theta-\gamma_1)}{\cos^2(\theta-\gamma_2)}=\frac{\mu(k_2)}{\mu(k_1)},
	\]
	which always exist because $\frac{\mu(k_2)}{\mu(k_1)}>0$ and $ \frac{\cos^2(\theta-\gamma_1)}{\cos^2(\theta-\gamma_2)} $ as a function of $ \theta $ maps onto $ (0,\infty) $. In fact this equation has several solutions in $[0,2\pi)$, so for simplicity we pick $ \theta^*_{\iota,\kappa} $ to be the smallest such solution. Then $ r^*_{\iota,\kappa} $ is simply given by
	\[ 
		r^*_{\iota,\kappa}=\frac{1}{\sqrt{-\cos^2(\theta^*_{\iota,\kappa}-\gamma_1)\mu_{\iota}(k_1)}}=\frac{1}{\sqrt{-\cos^2(\theta^*_{\iota,\kappa}-\gamma_2)\mu_{\kappa}(k_1)}}.
	\]
	\Cref{assumption:first} (i) follows by setting $ \hat{\bm{\eta}}_{\iota}(k_1)=\Sigma^{-1/2}(k_1)\bm{\xi}_{\iota}(k_1) $ and\\  $ \hat{\bm{\eta}}_{\kappa}(k_2)=\Sigma^{-1/2}(k_2)\bm{\xi}_{\kappa}(k_2) $ giving
	\begin{align*} 
		A(\bm{\tau}^*_{\iota,\kappa},\bm{k}_1)\hat{\bm{\eta}}_{\iota}(k_1)&=(\Sigma(k_1)+(r^*_{\iota,\kappa})^2\cos^2(\theta^*_{\iota,\kappa}-\gamma)B(k_1))\hat{\bm{\eta}_{\iota}}(k_1)\\
		&=-\frac{1}{\mu_{\iota}(k_1)}\Sigma^{1/2}(k_1)(\Sigma^{-1/2}(k_1)B(k_1)\Sigma^{-1/2}(k_1)-\mu_{\iota}(k_1)I)\hat{\bm{\xi}_{\iota}}(k_1)\\
		&=0,
	\end{align*}
	and similarly for $ \hat{\bm{\eta}}_\kappa(k_2) $. Moreover the kernels of $ A(\bm{\tau}^*_{\iota,\kappa},\bm{k}_1) $ and $ A(\bm{\tau}^*_{\iota,\kappa},\bm{k}_2) $ are both one dimensional due to \cref{eq:Aker1dim}.
	
	For \cref{assumption:first} (ii) we calculate
	\[ 
		\det \begin{pmatrix}
			\nu_{1,1} &\nu_{1,2}\\
			\nu_{2,1} &\nu_{2,2}\\
		\end{pmatrix}=\frac{4}{r^*}(\tan(\theta^*-\gamma_2)-\tan(\theta^*-\gamma_1))[\hat{\bm{\eta}}(k_1)\cdot \Sigma(k_1)\hat{\bm{\eta}}(k_1)][\hat{\bm{\eta}}(k_2)\cdot \Sigma(k_2)\hat{\bm{\eta}}(k_2)],
	\]
	which is zero only if $ \tan(\theta^*-\gamma_2)= \tan(\theta^*-\gamma_1) $. However, this would require
	\[ 
		\gamma_1-\gamma_2=n\pi
	\]
	for some $ n\in \mathbb{Z} $, but this contradicts the fact that $ \bm{k}_1 $ and $ \bm{k}_2 $ are linearly independent.
\end{proof}
Under stricter assumptions on the lattice $\Lambda'$ we can also prove \Cref{assumption:first} $ (iii) $ is satisfied. At least if we modify $\bm{\sigma}$ by an arbitrarily small amount. On the other hand we can drop the assumption that the vorticity is small for the following proposition.
\begin{proposition}\label{prop:As(III)Irrotational}
	If $ \bm{\sigma}=\bm{\sigma}_0 $, $ \Lambda'$ is a non-degenerate symmetric lattice, and $ \bm{\tau}^*\in\mathfrak{Z} $ is such that parts \textit{(i)} and \textit{(ii)} of \Cref{assumption:first} are satisfied, then for every $ \varepsilon>0 $ there exist $ \tilde{\bm{\sigma}} $ and corresponding $ \tilde{\bm{\tau}}^*\in\mathfrak{Z} $, such that $ \vert\tilde{\bm{\sigma}}-\bm{\sigma}_0\vert<\varepsilon $ and \Cref{assumption:first} is satisfied for $ \tilde{\bm{\sigma}} $ and $ \tilde{\bm{\tau}}^* $.
\end{proposition}
\begin{proof} 
	First note that we always have $\det A(\bm{\tau},0)\neq 0$. For $ \bm{k}\neq 0 $ can define $ A_q(\bm{\tau},\bm{k}) $ by replacing $ \bm{\sigma}_0 $ with $\bm{\sigma}_q$, where
	\[ 
	(\sigma_q)_j=\frac{(1+q)(\sigma_0)_jk_1^2+q g(\rho_j-\rho_{j+1})}{k_1^2}, \qquad j=1,\ldots,n.
	\]
	We also define 
	\[ 
		\bm{\tau}^*_q=((1+q)^{1/2}r^*,\theta^*)
	\]
	and obtain
	\[ 
		A_q(\bm{\tau}^*_q,\bm{k}_i)=(1+q)A_0(\bm{\tau}^*_0,\bm{k}_i), \qquad i=1,2.
	\]
	Clearly $ \bm{\tau}^*_0=\bm{\tau}^* $ and $ A_0(\bm{\tau},\bm{k})=A(\bm{\tau},\bm{k}) $, so $ \det A_q(\bm{\tau}^*_q,\bm{k}_i)=0 $ for $ i=1,2 $ and all $q\in [-1,\infty)$. On the other hand, if there exists some $ \bm{k}_3\in\Lambda' $ such that $ \det A(\bm{\tau}^*,\bm{k}_3)=0 $, then
	\[
		A_q(\bm{\tau}^*_q,\bm{k}_3)=(1+q)A_0(\bm{\tau}^*_0,\bm{k}_i)+ q\left(\frac{k_3^2}{k_1^2}-1\right)G,
	\]
	where $ G $ is the diagonal matrix with entries given by
	\[ 
		G_{j,j}=g(\rho_j-\rho_{j+1}).
	\]
	Now $ \det A_q(\bm{\tau}^*_q,\bm{k}_3)$ is a polynomial with respect to $ q $. Moreover, it is not identically $ 0 $; for example,
	\[ 
		\det A_{-1}(\bm{\tau}^*_{-1},\bm{k}_3)= \left(1-\frac{k_3^2}{k_1^2}\right)^n\det R,
	\]
	whence $ \det A_q(\bm{\tau}^*_q,\bm{k}_3)\neq 0 $ for all but a finite number of $ q\in [-1,\infty) $. Thus we can find $ \epsilon $ such that $ \det A_q(\bm{\tau}^*_q,\bm{k}_3)\neq 0$ for all $ q\in(-\epsilon,0)\cup(0,\epsilon) $.  Since only a finite number of $ \bm{k}\in \Lambda' $ can satisfy $ \det A(\bm{\tau}^*,\bm{k})=0 $, we repeat this a finite number of times, shrinking $ \epsilon $ if necessary. Moreover, by continuity, \Cref{assumption:first} (ii) remains satisfied for all $ q $ in some neighborhood of $ 0 $. Thus, after possibly shrinking $ \epsilon $ again, we find that \Cref{assumption:first} is satisfied for $ \tilde{\bm{\sigma}}=\bm{\sigma}_q $ if we choose $ \tilde{\bm{\tau}}=\bm{\tau}_q $ for any $ q\in(-\epsilon,0)\cup(0,\epsilon) $. Choosing any $ q $ such that $ \vert q\vert  $ is sufficiently small gives the desired result.
\end{proof}
We can now extend these results to small $ \vert \bm{\alpha}\vert $.
	\begin{proposition}\label{prop:small_alpha} If \Cref{assumption:first} holds for some $ \bm{\tau}^*\in\mathfrak{Z} $ with $ \bm{\alpha}=0 $, then there exist $ \bm{\tau}(\bm{\alpha})\in\mathfrak{Z} $ for every $\vert \bm{\alpha}\vert\ll 1 $ such that \Cref{assumption:first} holds for $\bm{\tau}(\bm{\alpha})$ as well.
	\end{proposition}
	\begin{proof}Defining
		\[ 
		\Delta_{a,b}\coloneqq \begin{pmatrix}
		\partial_r \det A(\bm{\tau}^*,\bm{k}_a) &\partial_\theta \det A(\bm{\tau}^*,\bm{k}_a)\\
		\partial_r \det A(\bm{\tau}^*,\bm{k}_b) &\partial_\theta \det A(\bm{\tau}^*,\bm{k}_b)
		\end{pmatrix}
		\]
		we can find
		\begin{align*} 
		\det \Delta_{1,2}=\frac{4\zeta_1^*\zeta_2^*}{r^*}\chi'_{\bm{k}_1}(\zeta_1^*)\chi'_{\bm{k}_2}(\zeta_2^*)(\tan(\theta^*-\gamma_2)-\tan(\theta^*-\gamma_1))
		\end{align*}
		where $ \zeta_i=r^2\cos^2(\theta-\gamma_i) $ and $ \chi_{\bm{k}_i}(\zeta_i)=\det A(\bm{\tau},\bm{k}_i) $.
		We have $ \det \Delta_{1,2} \neq 0$, because $ \chi_{\bm{k}_i}(\zeta_i)\neq 0 $, due to the fact that $ \chi_{\bm{k}_i} $ is a polynomial and $ \zeta^*_i $ is a root with multiplicity $ 1 $ and, as above, $ \tan(\theta^*-\gamma_2)-\tan(\theta^*-\gamma_1)=0 $ only if $ \gamma_1-\gamma_2=n\pi $.
		Since $ \det\Delta_{1,2}\neq 0 $ we can define $ \bm{\tau}(\bm{\alpha}) $ such that $ \bm{\tau}(0)=\bm{\tau}^* $ and $ \det A(\bm{\tau}(\bm{\alpha}),\bm{k}_1) =\det A(\bm{\tau}(\bm{\alpha}),\bm{k}_2)=0 $ for all $ \vert \bm{\alpha}\vert\ll 1$.
		This means \Cref{assumption:first} (i) remains true. It is clear that part (ii) remains true by continuity. The same is true for part (iii) and any finite set of $ \bm{k}_i\in \Lambda' $, but since $ \det A(\bm{\tau}^*,\bm{k})\sim \vert \bm{k}\vert^{2n} $ for all large $ \vert \bm{k}\vert $ this is sufficient.
	\end{proof}
	Combining \Crefrange{prop:As(I&II)Irrotational}{prop:small_alpha} with the existence result in \Cref{thm:exist} gives \Cref{thm:existssmallvort}.
	\subsection{Discussion of large vorticity}\label{sec:largeVorticity}
	This subsection does not include any definite results about existence of solutions, but rather exemplifies some of the possibilities that can occur if we relax the assumption $ \vert \bm{\alpha}\vert\ll 1 $. We also exclusively focus on part \textit{(i)} and \textit{(ii)} of \Cref{assumption:first} because both and part \textit{(iii)} can reasonably be assumed true for most parameter values. There is no assumption on $\bm{\alpha}$ in \Cref{prop:As(III)Irrotational} and it is not unreasonable to suspect a similar result to hold true even for general $\Lambda'$. In the discussion below we implicitly consider part \textit{(iii)} of \Cref{assumption:first} to be satisfied when mentioning `solutions' or `interfaces'.
	
	In the general case we can write the matrix $ A(\bm{\tau},\bm{k}) $ as
	\[ 
		 A(\bm{\tau},\bm{k})=\Sigma(k)+r^2(C(\bm{k},\theta)B(k)C(\bm{k},\theta)+C(\bm{k},\theta)DS(\bm{k},\theta)),
	\]
	where
	\begin{align*}
		C_{j,j}(\bm{k},\theta)&=\beta_j(\bm{k})=\cos(\theta_j-\alpha_jd_j-\gamma),\\
		S_{j,j}(\bm{k},\theta)&=\beta^\perp_j(\bm{k})=\sin(\theta_j-\alpha_jd_j-\gamma),\\
		D_{j,j}&=\rho_{j}\alpha_{j}-\rho_{j+1}\alpha_{j+1}.
	\end{align*}
	From this we can define the matrix
	\[ 
		R(\bm{k},\theta)\coloneqq\Sigma^{-1/2}(k)[C(\bm{k},\theta)B(k)C(\bm{k},\theta)+C(\bm{k},\theta)DS(\bm{k},\theta)]\Sigma^{-1/2}(k),
	\]
	which is symmetric, whence it has $ n $ linearly independent eigenvectors $ \bm{\xi}_\iota(\bm{k},\theta) $, $ \iota=1,\ldots,n $ with corresponding real eigenvalues $ \mu_\iota(\bm{k},\theta) $, $ \iota=1,\ldots,n $. We can find $ \bm{\tau}^* $ such that \cref{assumption:first} \textit{(i)} is satisfied by finding $ \theta^* $ such that $ \mu_\iota(\bm{k}_1,\theta^*)=\mu_\kappa(\bm{k}_2,\theta^*)<0 $ for some $ 1\leq \iota,\kappa\leq n $. This is because
	\[ 
		R(\bm{k},\theta)\bm{\xi}_\iota(\bm{k},\theta)=\mu_\iota(\bm{k},\theta)\bm{\xi}_\iota(\bm{k},\theta)
	\]
	implies
	\[ 
		A(\bm{\tau},\bm{k})\hat{\bm{\eta}}_\iota(\bm{k},\theta)=0, 
	\]
	for $\hat{\bm{\eta}}_\iota(\bm{k},\theta)=\Sigma^{-1/2}(k)\bm{\xi}_\iota(\bm{k},\theta)$ if
	\begin{equation}\label{eq:rdef} 
		(r_\iota(\bm{k},\theta))^2=-\frac{1}{\mu_\iota(\bm{k},\theta)}.
	\end{equation}
	Conversely it is not difficult to check that a vector in $\ker A(\bm{\tau},\bm{k})$ gives rise to an eigenvector of $R(\bm{k},\theta)$ with negative eigenvalue given by \cref{eq:rdef}. So \Cref{assumption:first} part $(i)$ is satisfied if and only if we have $ \theta^* $ such that $ \mu_\iota(\bm{k}_1,\theta^*)=\mu_\kappa(\bm{k}_2,\theta^*)<0 $. If, in addition, $\partial_{\theta}\mu_\iota(\bm{k}_1,\theta^*)\neq \partial_{\theta}\mu_\kappa(\bm{k}_2,\theta^*)$ then \Cref{assumption:first} part \textit{(ii)} is also satisfied. We can see this by considering
	\[ 
		0=\frac{d}{d\theta}(\hat{\bm{\eta}}_\iota(\bm{k}_1,\theta)\cdot A((r_\iota(\bm{k}_1,\theta),\theta),\bm{k})\hat{\bm{\eta}}_\iota(\bm{k}_1,\theta))\vert_{\theta=\theta^*}=\nu_{1,2}+\nu_{1,1}\frac{(r^*)^3}{2}\partial_{\theta}\mu_\iota(\bm{k}_1,\theta^*),
	\]
	so  $\nu_{1,2}=-\nu_{1,1}\frac{(r^*)^3}{2}\partial_{\theta}\mu_\iota(\bm{k}_1,\theta^*)$. Similarly we obtain $\nu_{2,2}=-\nu_{2,1}\frac{(r^*)^3}{2}\partial_{\theta}\mu_\kappa(\bm{k}_2,\theta^*)$, which means
	\[
	\det \begin{pmatrix}
	\nu_{1,1} & \nu_{1,2}\\
	\nu_{2,1} & \nu_{2,2}
	\end{pmatrix}=\nu_{1,1}\nu_{2,1}\frac{(r^*)^3}{2}(\partial_{\theta}\mu_\iota(\bm{k}_1,\theta^*)-\partial_{\theta}\mu_\kappa(\bm{k}_2,\theta^*)).
	\]
	It is not difficult to show that $\nu_{1,1}\nu_{2,1}\frac{(r^*)^3}{2}\neq 0$, whence \Cref{assumption:first} part \textit{(ii)} is satisfied if $\partial_{\theta}\mu_\iota(\bm{k}_1,\theta^*)\neq \partial_{\theta}\mu_\kappa(\bm{k}_2,\theta^*)$. In summary, if there exists $\theta^*$ such that
	\[
		 \mu_\iota(\bm{k}_1,\theta^*)=\mu_\kappa(\bm{k}_2,\theta^*)<0 \quad\text{and}\quad\partial_{\theta}\mu_\iota(\bm{k}_1,\theta^*)\neq \partial_{\theta}\mu_\kappa(\bm{k}_2,\theta^*),
	\]
	then parts \textit{(i)} and \textit{(ii)} of \Cref{assumption:first} are satisfied for
	\[
		\bm{\tau}_{\iota,\kappa}^*=\left(\frac{1}{\sqrt{-\mu_\iota(\bm{k}_1,\theta^*)}},\theta^*\right)=\left(\frac{1}{\sqrt{-\mu_\kappa(\bm{k}_2,\theta^*)}},\theta^*\right)
	\]
	Moreover, if we have such an intersection point $\theta^*$ between $\mu_\iota(\bm{k}_1,\theta)$ and $\mu_\kappa(\bm{k}_2,\theta)$, then by \Cref{thm:exist} the solution to \crefrange{eq:curl}{eq:int_cond} will have interface profiles given by
	\begin{equation}\label{eq:interfaceGen}
	\bm{\eta}_{\iota,\kappa}(\theta^*)=\hat{\bm{\eta}}_\iota(\bm{k}_1,\theta^*)\cos(\bm{k}_1\cdot\bm{x}')t_1+\hat{\bm{\eta}}_\kappa(\bm{k}_2,\theta^*)\cos(\bm{k}_1\cdot\bm{x}')t_2+\mathcal{O}(\vert\bm{t}\vert^2)
	\end{equation}
	for some $\bm{t}=(t_1,t_2)\in B_\epsilon(0) $.
	
	\Cref{thm:existssmallvort} gives us the existence of at least $ n^2 $ such intersection points in the case $ \vert \bm{\alpha}\vert\ll 1 $, that is, $ \mu_\iota(\bm{k}_1,\theta) $ intersects $ \mu_\kappa(\bm{k}_2,\theta) $ in this manner for all $ 1\leq \iota, \kappa\leq n $. In fact, there will in general be more intersections points. Since $ R(\bm{k},\theta) $ is $ \pi $-periodic we can expect an even number of intersection points on the interval $ [0,\pi) $; see \Cref{subfig:EigvalsAlph0}. However, for $ \vert \bm{\alpha}\vert= 0 $ we also know that all the additional intersection points corresponds to a surface profile that is identical to one of the original $n^2$. This is because $\bm{\xi}_{\iota}(\bm{\kappa},\theta^*)$ is completely independent of $\theta^*$ in this case and $\bm{\eta}_{\iota,\kappa}(\theta^*)$ in \cref{eq:interfaceGen} is the same as $\bm{\eta}_{\iota,\kappa}$ in \Cref{thm:existssmallvort}. For $\bm{\alpha}\neq 0 $ this is not necessarily true, but for $\vert \bm{\alpha}\vert\ll 1 $ we have at least $n^2$ different solutions and the surface profile corresponding to any additional intersection point is very similar to the surface profile of one of these $n^2$ solutions. For larger $ \vert \bm{\alpha}\vert $ we can no longer guarantee the same amount of intersection points. What we can say is that if $ k_1=k_2 $, then $ \mu_\iota(\bm{k}_1,\theta) $ and $ \mu_\iota(\bm{k}_2,\theta) $ are the same function shifted by the angle between $ \bm{k}_1 $ and $ \bm{k}_2 $. This means $ \mu_\iota(\bm{k}_1,\theta) $ and $ \mu_\iota(\bm{k}_2,\theta) $ intersects if $ \Lambda' $ is a symmetric lattice. Moreover, if $ \bm{\alpha} $ and $ \bm{\rho} $ are such that $ D=0 $ then $ R(\bm{k},\theta) $ is negative semi-definite so the intersection points will also almost surely satisfy $\mu_\iota(\bm{k}_1,\theta^*)=\mu_\iota(\bm{k}_2,\theta^*)<0$. This means we should at least have solutions with surface profiles $\bm{\eta}_{\iota,\iota}$, $1\leq \iota\leq n$ in this case; see \Cref{subfig:EigvalsAlphNon0}. It should also be noted that when $ \vert \bm{\alpha}\vert $ is not small, all different intersection points can correspond to more substantially different eigenvectors and thus different surface profiles. Finally, when $ D\neq 0 $ the eigenvalues can be positive and one or more of the intersection points may satisfy $ \mu_\iota(\bm{k}_1,\theta^*)=\mu_\kappa(\bm{k}_2,\theta^*)\geq 0 $, which means we cannot find a corresponding real $ r^* $ satisfying \cref{eq:rdef}; see \Cref{subfig:EigvalsAlphNonEq}. If this is the case for all intersection points, then there are no $ \bm{\tau}^* $ satisfying \Cref{assumption:first} and we cannot obtain any solutions with \Cref{thm:exist}; see \Cref{subfig:EigvalsAlphNoSol}. It should be noted that this does not exclude small amplitude solutions to \crefrange{eq:curl}{eq:int_cond}. For example we could find solutions that are constant in one horizontal direction when $A(\bm{\tau}^*,\bm{k})$ has a one dimensional kernel, giving so called $2$\sfrac{1}{2}-dimensional waves.
 	\begin{figure}[h!]
		\begin{centering}
			\begin{subfigure}{0.45\textwidth}
				\includegraphics[scale=0.47]{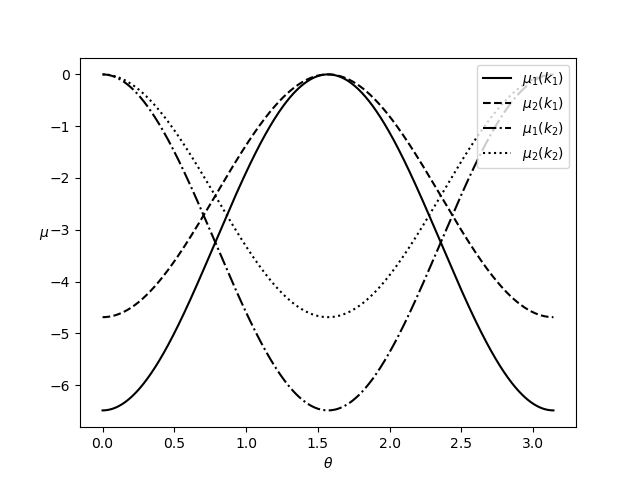}
				\caption{The eigenvalues of $ R(\bm{k}_1,\theta) $ and $ R(\bm{k}_2,\theta) $, for $ n=2 $ and $ \bm{\alpha}=0 $, plotted as functions of $ \theta $. Here $ \mu_\iota(\bm{k}_1,\theta) $ intersects $ \mu_\kappa(\bm{k}_2,\theta) $ for any $ 1\leq \iota,\kappa\leq n $. The corresponding solutions have surface profiles $\bm{\eta}_{1,1}(\theta^*)$, $\bm{\eta}_{1,2}(\theta^*)$, $\bm{\eta}_{2,1}(\theta^*)$ and $\bm{\eta}_{2,2}(\theta^*)$ for two different $\theta^*$ each, but $\bm{\alpha}=0$ so $\bm{\eta}_{\iota,\kappa}(\theta^*_1)=\bm{\eta}_{\iota,\kappa}(\theta^*_2)$.}
				\label{subfig:EigvalsAlph0}
			\end{subfigure}
			\hfill
			\begin{subfigure}{0.45\textwidth}
				\includegraphics[scale=0.47]{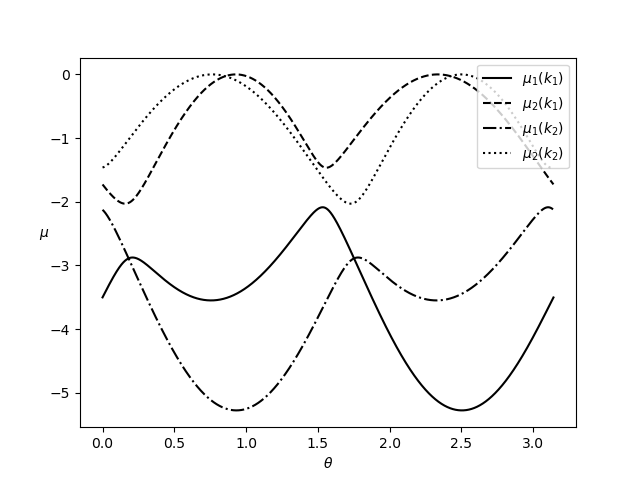}
				\caption{The eigenvalues of $ R(\bm{k}_1,\theta) $ and $ R(\bm{k}_2,\theta) $, for $ n=2 $ and $ \bm{\alpha}\neq 0 $ such that $ D=0 $, plotted as functions of $ \theta $. Here there are no intersection points  between $ \mu_\iota(\bm{k}_1,\theta) $ and $ \mu_\kappa(\bm{k}_2,\theta) $ for $ \iota\neq \kappa $. The corresponding solutions have surface profiles $\bm{\eta}_{1,1}(\theta^*)$ and $\bm{\eta}_{2,2}(\theta^*)$ for two different $\theta^*$ each.}
				\label{subfig:EigvalsAlphNon0}
			\end{subfigure}\\
			\begin{subfigure}{0.45\textwidth}
				\includegraphics[scale=0.47]{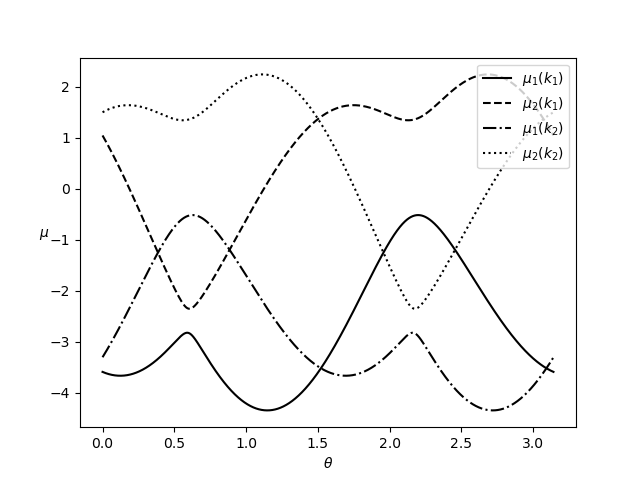}
				\caption{The eigenvalues of $ R(\bm{k}_1,\theta) $ and $ R(\bm{k}_2,\theta) $, for $ n=2 $ and $ \bm{\alpha}\neq 0 $ such that $ D\neq 0 $, plotted as functions of $ \theta $. Here $ \mu_\iota(\bm{k}_1,\theta) $ intersects $ \mu_\kappa(\bm{k}_2,\theta) $ for any $ 1\leq \iota,\kappa\leq n $, but the intersections between $ \mu_2(\bm{k}_1,\theta) $ and $ \mu_2(\bm{k}_2,\theta) $ lie above $ 0 $. The corresponding solutions have surface profiles $\bm{\eta}_{1,1}(\theta^*)$, $\bm{\eta}_{1,2}(\theta^*)$ and $\bm{\eta}_{2,1}(\theta^*)$ for two different $\theta^*$ each.}
				\label{subfig:EigvalsAlphNonEq}
			\end{subfigure}
			\hfill
			\begin{subfigure}{0.45\textwidth}
				\includegraphics[scale=0.47]{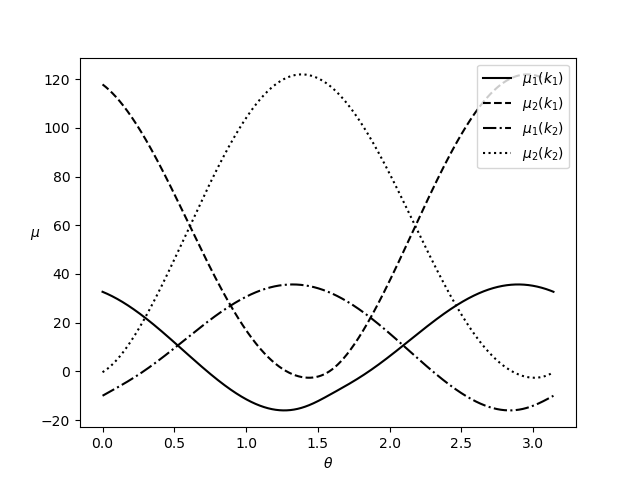}
				\caption{The eigenvalues of $ R(\bm{k}_1,\theta) $ and $ R(\bm{k}_2,\theta) $, for $ n=2 $ and $ \bm{\alpha}\neq 0 $ such that $ D\neq 0 $, plotted as functions of $ \theta $. Here $ \mu_\iota(\bm{k}_1,\theta) $ intersects $ \mu_\kappa(\bm{k}_2,\theta) $ for any $ 1\leq \iota,\kappa\leq n $, but all intersections lie above $ 0 $, so there are no $ \bm{\tau}^* $ satisfying \Cref{assumption:first}. Thus we cannot obtain any solutions with \Cref{thm:exist} in this case.}
				\label{subfig:EigvalsAlphNoSol}
			\end{subfigure}
		\end{centering}
		\caption{In this figure we illustrate the eigenvalues of $ R(\bm{k}_1,\theta) $ and $ R(\bm{k}_2,\theta) $ for $ n=2 $ and four different choices of $ \bm{\alpha} $.}
		\label{fig:Eigvals}
	\end{figure}
\newpage
\appendix
\section{A multi parameter bifurcation theorem}\label{sec:AppendixBif}
Let $\mathcal{C}=\mathbb{R}^n$, $\mathcal{X}$ and $\mathcal{Y}$ be Banach spaces, and $\mathcal{X}_i=\spann\{x_i\}$ and $\mathcal{Y}_i=\spann\{y_i\}$, $i=1,\ldots,n$ be one dimensional subspaces of $\mathcal{X}$ and $\mathcal{Y}$ respectively.
This means that we can write
\begin{align*}
\mathcal{X}&=\left(\bigoplus_{i=1}^n \mathcal{X}_i\right)\oplus \tilde{\mathcal{X}},\\
\mathcal{Y}&=\left(\bigoplus_{i=1}^n \mathcal{Y}_i\right)\oplus \tilde{\mathcal{Y}},
\end{align*}
where $\tilde{\mathcal{X}}$ and $\tilde{\mathcal{Y}}$ are closed subspaces.
This decomposition allows us to define the projections $Q_i$ and $P_i$, $i=1,\ldots,n$, which are projections onto $\mathcal{X}_i$ and $\mathcal{Y}_i$ along
\begin{align*}
\hat{\mathcal{X}}_i\oplus \tilde{\mathcal{X}}&\coloneqq\left(\bigoplus_{
	\begin{subarray}
	kj=1\\
	j\neq i
	\end{subarray}
}^n 
\mathcal{X}_j\right)\oplus \tilde{\mathcal{X}}\\
\intertext{and}
\hat{\mathcal{Y}}_i\oplus \tilde{\mathcal{Y}}&\coloneqq \left(\bigoplus_{
	\begin{subarray}
	kj=1\\
	j\neq i
	\end{subarray}
}^n 
\mathcal{Y}_j\right)\oplus \tilde{\mathcal{Y}}\\
\end{align*}
respectively.
Moreover, in this section we let $P=\sum_{i=1}^n P_i$ and $Q=\sum_{i=1}^n Q_i$.
Below is a bifurcation result to solve an equation of the form
\begin{equation}\label{eq:Bifurcation_equation}
	F[x,\bm{c}]=0,
\end{equation}
for an operator $ F:\mathcal{X}\times \mathcal{C}\to \mathcal{Y} $. When $ n=1 $ this result coincides with the Crandall-Rabinowitz bifurcation theorem.
\begin{theorem}\label{thm:bif}
	If $F\in C^k(\mathcal{X}\times\mathcal{C},\mathcal{Y})$, with $k\in \mathbb{Z}$, such that $ k\geq 2 $, $ k=\infty $ or $ k=\omega $ (that is $ F $ is analytic), is an operator with the following properties:
	\begin{itemize}
		\item[(i)] $F[0,\bm{c}]=0$ for all $\bm{c}\in \mathcal{C}$ and there exists $\bm{c}^*\in \mathcal{C}$ such that $D_1F[0,\bm{c}^*]:\mathcal{X}\to \mathcal{Y}$ is a Fredholm operator of index $0$
		\item[(ii)] The kernel of $D_1F[0,\bm{c}^*]$ is $n$-dimensional and given by $\bigoplus_{i=1}^n \mathcal{X}_i$.
		\item[(iii)] If $P_iD_1D_{j+1}F[0,\bm{c}^*](x_i,c_j-c_j^*)=\nu_{ij}(c_j-c^*_j)y_i$, then the matrix $\bm{\nu}$ given by
		\[
		(\bm{\nu})_{i,j}=\nu_{ij}
		\]
		is invertible.
		\item[(iv)] There exist closed subspaces $\tilde{\mathcal{X}}_i$ of $\tilde{\mathcal{X}}$ and $\tilde{\mathcal{Y}}_i$ of $\tilde{\mathcal{Y}}$ for each $i=1,\ldots,n$ such that
		\begin{align*}
		F(\hat{\mathcal{X}}_i\oplus \tilde{\mathcal{X}_i},\bm{c})\subseteq \hat{\mathcal{Y}}_i\oplus\tilde{\mathcal{Y}}_i,
		\end{align*}
		and
		\begin{align*}
		(I-P_i)D_1F[0,\bm{c}^*]\vert_{\hat{\mathcal{X}}_i\oplus \tilde{\mathcal{X}}_i}:\hat{\mathcal{X}}_i\oplus \tilde{\mathcal{X}}_i\to   \hat{\mathcal{Y}}_i\oplus \tilde{\mathcal{Y}}_i
		\end{align*}
		is a Fredholm operator of index $0$ with kernel $\hat{\mathcal{X}}_i$.
	\end{itemize}
	Then there exists an $\epsilon$ such that for every $\bm{s}=(s_1,\ldots,s_n)\in B_\epsilon(0)=\{\bm{s}\in\mathbb{R}^n:\vert \bm{s}\vert<\epsilon\}$ equation \eqref{eq:Bifurcation_equation} has a solution $(x[\bm{s}],\bm{c}[\bm{s}])$.
	Moreover, $(x[\cdot],\bm{c}[\cdot])\in C^{k-1}(B_\epsilon(0),\mathcal{X}\times \mathcal{C})$ and
	\[
	x[\bm{s}]=\sum_{i=1}^n s_ix_i+o(\vert \bm{s}\vert),\qquad \bm{c}[\bm{s}]=\bm{c}^*+\mathcal{O}(\vert \bm{s}\vert).
	\]
\end{theorem}
\begin{remark} The conditions \textit{(i)},\textit{(ii)}, and \textit{(iii)} are clear analogues to the standard local bifurcation theorem by Crandall \& Rabinowitz.
	Condition \textit{(iv)} has no analogue because it is superfluous if $n=1$.
	In the case when $n\geq 2$ it is necessary unless we also relax the conclusion of the theorem. Moreover, separates the domain and codomain of the operators in a way that makes the proof of the theorem quite straightforward.
\end{remark}
\begin{proof} We begin by performing a Lyapunov Schmidt reduction.
	Writing $x=\sum_{i=1}^n s_ix_i+\tilde{x}$ where $s_ix_i=Q_ix$ and $\tilde{x}=(I-Q)x$ we obtain the equations
	\begin{align*}
	(I-P)F\left[\sum_{j=1}^n s_jx_j+\tilde{x},\bm{c}\right]&=0,\\
	P_iF\left[\sum_{j=1}^n s_jx_j+\tilde{x},\bm{c}\right]&=0&& \qquad i=1,\ldots,n.
	\end{align*}
	By assumption \textit{(i)} we can apply the implicit function theorem to obtain $\tilde{x}[\bm{s},\bm{c}]$ that solves the first equation.
	We note that $\tilde{x}[0,\bm{c}]=0$ and $\partial_{s_i}\tilde{x}[0,\bm{c}^*]=0$, $i=1,\ldots,n$.
	Moreover, by assumption \textit{(iv)} we can consider the restricted operator $F:\hat{\mathcal{X}}_i\oplus \tilde{\mathcal{X}}_i\to \hat{\mathcal{Y}}_i\oplus \tilde{\mathcal{Y}}_i$ and perform a Lyapunov-Schmidt reduction.
	It follows that  $\tilde{x}[\bm{s},\bm{c}]\vert_{s_i=0}\in \hat{\mathcal{X}}_i\oplus \tilde{\mathcal{X}}_i$.
	Due to this fact and assumption \textit{(iv)} we get that
	\[
	P_iF\left.\left[\sum_{j=1}^n s_jx_j+\tilde{x}[\bm{s},\bm{c}],\bm{c}\right]\right\vert_{s_i=0}=0
	\]
	This means that we can write
	\[
	P_iF\left[\sum_{j=1}^n s_jx_j+\tilde{x}[\bm{s},\bm{c}],\bm{c}\right]=s_iH_i(\bm{s},\bm{c})
	\]
	and solve
	\begin{equation}\label{eq:bif_fin_dim_eqs}
	H_i(\bm{s},\bm{c})=0,\qquad i=1,\ldots,n.
	\end{equation}
	instead. 
	The functions $H_i$ are differentiable and 
	\begin{align*}
	H_i(0,\bm{c}^*)&=P_iD_1 F[0,\bm{c}^*](x_i+\partial_{s_i}\tilde{x}[0,\bm{c}^*])=0,\\
	\partial_{c_j}H_i(0,\bm{c}^*)(c_j-c_j^*)&=P_iD_1 F[0,\bm{c}^*](\partial_{s_i}\partial_{c_j}\tilde{x}[0,\bm{c}^*](c_j-c_j^*))\\
	&\qquad+P_iD_1D_{j+1} F[0,\bm{c}^*](x_i+\partial_{t_i}\tilde{x}(0,\bm{c}^*),c_j-c_j^*)\\
	&=P_iD_1D_{j+1} F[0,\bm{c}^*](x_i,c_j-c_j^*)\\
	&=\nu_{ij}(c_j-c^*_j)y_i.
	\end{align*}
	Since $\bm{\nu}$ is invertible this means we can apply the implicit function theorem to obtain a differentiable function $\bm{c}(\bm{t})$ defined in $B_{\epsilon}(0)$ that solves \cref{eq:bif_fin_dim_eqs}.
	
	We end by noting that we do not obtain uniqueness due to the fact that if $s_i=0$ then $P_iF(x,\bm{c})=0$ is not equivalent to $H_i(\bm{s},\bm{c})=0$.
\end{proof}
\section*{acknowledgements} This work was carried out during the tenure of an ERCIM 'Alain Bensoussan' Fellowship Programme. The Author would also like to thank Mats Ehrnström for constructive comments on the structure of the article.
\bibliographystyle{siam}
\bibliography{InternalWavesWithVorticity.bib}
\end{document}